\documentclass[12pt]{amsart}

\numberwithin{equation}{section}

\usepackage{amsfonts,amssymb,amsmath,amsthm}
\usepackage{url}
\usepackage{enumerate}
\usepackage{color}
\usepackage{algorithmic}
\usepackage{algorithm}
\usepackage[colorlinks]{hyperref}

\theoremstyle{plain}
\newtheorem{theorem}{Theorem}[section]
\newtheorem{corollary}[theorem]{Corollary}
\newtheorem{proposition}[theorem]{Proposition}
\newtheorem{lemma}[theorem]{Lemma}

\theoremstyle{definition}
\newtheorem{definition}[theorem]{Definition}

\theoremstyle{remark}
\newtheorem{remark}[theorem]{Remark}
\newtheorem{example}[theorem]{Example}

\begin{document}

\def\gb{{\mathfrak b}}
\def\ra{\rightarrow}
\def\gm{{\mathfrak m}}

\title[From grids to pseudo-grids of lines: resolution and seminormality]{From grids to pseudo-grids of lines: resolution and seminormality}

\author{Francesca Cioffi}\thanks{Corresponding author: Francesca Cioffi, Dipartimento di Matematica e Applicazioni \lq\lq R. Caccioppoli\rq\rq, Universit\`{a} degli Studi di Napoli Federico II, Via Cintia, 80126 Napoli, Italy, cioffifr@unina.it}
\address{Dipartimento di Matematica e Applicazioni \lq\lq R. Caccioppoli\rq\rq, Universit\`{a} degli Studi di Napoli Federico II, Via Cintia, 80126 Napoli, Italy}
\email{cioffifr@unina.it}

\author{Margherita Guida}
\address{Dipartimento di Matematica e Applicazioni \lq\lq R. Caccioppoli\rq\rq, Universit\`{a} degli Studi di Napoli Federico II, Via Cintia, 80126 Napoli, Italy}
\email{maguida@unina.it}

\author{Luciana Ramella}
\address{Dipartimento di Matematica, Universit\`{a} degli Studi di Genova, Via Dodecaneso, 35 16146 Genova  Italy}
\email{ramella@dima.unige.it}


\begin{abstract}
Over an infinite field $K$, we investigate the minimal free resolution of some configurations of lines. We explicitly  describe the minimal free resolution of {\em complete grids of lines} and obtain an analogous result about the so-called {\em complete pseudo-grids}. Moreover, we characterize the total Betti numbers of configurations that are obtained posing a multiplicity condition on the lines of either a complete grid or a complete pseudo-grid. Finally, we analyze  when a complete pseudo-grid is seminormal, differently from a complete grid. 
The main tools that have been involved in our study are the mapping cone procedure and properties of liftings, of pseudo-liftings and of weighted ideals. 

Although complete grids and pseudo-grids are  hypersurface configurations and many results about such type of configurations have already been stated in literature, we give new contributions, in particular about the maps of the resolution.
\end{abstract}

\maketitle


\section{Introduction}

In this paper, we deal with the problem of describing the minimal free resolution of configurations of lines in a projective space $\mathbb P^n_K$ of any dimension $n$ over an infinite field $K$.  
We focus on {\em complete grids of lines}, which are particular configurations of projective lines that have been introduced in \cite{GO2007} and are obtained as the projective closure of affine lines that are parallel to the coordinate axes and pass through a lattice of points. 

A complete grid of lines is a reduced projective variety with nice properties: it is Cohen-Macaulay and its defining ideal is generated by products of linear forms. More generally, configurations of linear varieties are not always Cohen-Macaulay or defined by products of linear forms (see \cite[Proposition 5.7]{BPS2005}). They have been extensively studied by many authors in several mathematical contexts (for motivations and a survey see \cite{SS2013} and the references therein, for example).

The study of configurations of linear varieties is also interesting when multiplicity conditions are given on the varieties. 
In this context, a complete grid of lines turns into a fat complete grid of lines if multiplicity conditions are posed on the lines of the grid. When the multiplicity conditions are the same for all the lines, we obtain a {\em $m$-fat complete grid of lines} (see \cite{G2014}).

Complete grids and $m$-fat complete grids are liftings of suitable monomial ideals, that we denote by $\Omega$ and $\Gamma$, respectively (see \cite{GOR2014}). When liftings are replaced by pseudo-liftings, we obtain {\em complete pseudo-grids} and {\em $m$-fat complete pseudo-grids}, in which the affine lines are no more parallel to the coordinate axes. 
The monomial ideal $\Gamma$ is the $m$-th symbolic power of the monomial ideal $\Omega$ and the ideal defining a $m$-fat complete grid (respectively, pseudo-grid) is the $m$-th symbolic power of the ideal defining the complete grid (respectively, pseudo-grid) on which the fat grid is supported, as well (see Proposition~\ref{prop:symbolic power}).

In general, it is not easy to calculate the minimal free resolution of a configuration of linear varieties.
In this paper, we explicitly describe the minimal free resolution of complete grids and pseudo-grids of lines in a projective space of any dimension $n$ by proving, indeed, a conjecture that has been posed in the paper \cite{Guida2008} for complete grids (see \cite{Guida2008, G2014} for the case $n=3$). We also obtain a description of the total Betti numbers of $m$-fat complete grids and pseudo-grids.

Our main result is that, over an infinite field $K$, the minimal free resolution of the defining ideal $I\subset R:=K[x_0,x_1,\dots,x_n]$ of either a complete grid or pseudo-grid of lines of type $(\ell_1,\dots,\ell_n)$ in $\mathbb P^n_K$ is $0 \rightarrow F_{n-2} \rightarrow F_{n-3} \rightarrow \cdots \rightarrow F_1 \rightarrow F_0 \rightarrow I \rightarrow 0$, where
\begin{equation}\label{eq:conjecture}
\begin{array}{lll}
F_0 &=& \oplus _{1 \leq i<j \leq n} R(-\ell_i -\ell_j),\\
F_1 &=& \oplus _{1 \leq i_1 < i_2 < i_3 \leq n} R^2 (-\ell_{i_1} -\ell_{i_2} -\ell_{i_3}),\\
&\vdots & \\
F_{n-3} &=& \oplus_{i=1, \ldots , n}  R^{n-2} (-\ell_1 - \cdots - \widehat{\ell_i} - \cdots - \ell_n),\\
F_{n-2} &=& R^{n-1} (-\ell_1 \ldots -\ell_n).
\end{array}
\end{equation} 
together with a complete description of the maps of this resolution (see Theorem~\ref{th:risoluzione Omega} and Corollary~\ref{cor:congettura}).
We prove this result by applying a mapping cone procedure to the above quoted monomial ideal $\Omega$. This mapping cone procedure provides the free resolution with modules in \eqref{eq:conjecture} (together with the maps), which turns out to be minimal due to properties of $\Omega$ and, more generally, of {\em weighted ideals} (see Section~\ref{sec:weighted ideals}, especially Theorem~\ref{th:betti weighted}). The syzygies of the complete grid can be even obtained by a standard Gr\"obner bases rewriting procedure, in alternative to the mapping cone procedure (see Remark~\ref{rem:dimostrazione indipendente}). 
It is also noteworthy that under certain general conditions pseudo-grids are seminormal, differently from grids (see Proposition \ref{prop:seminormality}).
  
For what concerns configurations of lines with multiplicity conditions, we know that a $m$-fat complete grid of lines in $\mathbb P^n_K$ is a lifting of the above quoted monomial ideal $\Gamma$, and show that it has the same total Betti numbers of a $m$-fat scheme of $n$ general points in $\mathbb P^{n-1}_K$ (Lemma \ref{lemma:generatori fat grid} and Proposition \ref{prop:betti fat grids}). An analogous result holds for $m$-fat complete pseudo-grids of lines, which are  pseudo-liftings of the ideal $\Gamma$.

We highlight that the monomial ideals $\Omega$ and $\Gamma$ are neither stable nor of decreasing type (see \cite{Sharifan}), so we cannot refer to already known results about such types of monomial ideals in order to compute minimal free resolutions and maps by a mapping cone procedure.
Nevertheless, the ideal $\Omega$ is strictly related to Stanley-Reisner ideals of matroids, so that we can identify complete grids and pseudo-grids of lines with particular configurations of lines contained in the larger family of hypersurface configurations, as defined in \cite{GHMN17}. In this paper we focus on lines, so now we only remark that, starting from a larger class of monomial ideals and even from more general liftings, in this family a natural generalization of our configurations of lines to configurations of linear varieties of dimension $k>1$ can be considered too, satisfying the same properties of the $1$-dimensional case. In particular, the description of the minimal free resolution with the maps can be extended to such generalization, taking also into account the results of \cite{Gal16,ELSW18} about the graded Betti numbers of the Stanley-Reisner ideals quoted above.

In this context, several results about the minimal free resolution of a hypersurface configuration are presented with different approaches in \cite[Theorem 3.4]{PS15}, \cite[Remark 2.11]{GHM13}, \cite[Lemma 3.1, Theorem~3.3(3) and Corollary 3.5]{GHMN17}.

The description of the minimal free resolutions for complete grids and pseudo-grids that we provide in our paper uses a mapping cone procedure that is different from those proposed in the papers above.

The paper is organized in the following way. After some preliminaries about minimal free resolutions that are recalled in Section \ref{sec:preliminaries}, we retrace definitions and results about liftings and complete grids in Sections \ref{sec:lifting} and \ref{sec:grids}. In Section \ref{sec:weighted ideals}, the notion of weighted ideal is studied and information are obtained about the Betti numbers of the monomial ideals $\Omega$ and $\Gamma$, and hence of the configurations of lines that we are considering. 
In Section \ref{sec:mapping cone}, we prove our main contribution. 
Finally, in Section \ref{sec:seminormality} we explain the seminormality of complete pseudo-grids of lines.

The examples that we provide have been computed by using CoCoa (see \cite{CoCoA}). 

\section{Some preliminaries and notation}
\label{sec:preliminaries}

For results about the computation of syzygy modules of polynomial modules, that go back to Janet, Schreyer, M\"oller, Mora, we refer to \cite[Sections 23.7 and 23.8]{Mora}, to the references therein, and to \cite[chapters 2 and 3]{KR1}.  
For definitions and results about Hilbert functions we refer to \cite{Ei,KR2}.

Recall that a free resolution of a homogeneous ideal $I\subset R=K[x_0,x_1,\dots,x_n]$ is an exact complex $\dots\rightarrow E_{k} {\buildrel {\delta_n} \over \longrightarrow } \dots {\buildrel {\delta_1} \over \longrightarrow } E_0 {\buildrel {\delta_0} \over \longrightarrow } I \longrightarrow 0$, where the modules $E_i=\oplus R^{\beta_{i,j}}(-j)$ are free and every integer $\beta_{i,j}$ is the cardinality of a set of non-necessarily minimal generators of degree $j$ for the $i$-th module of syzygies of $I$. The integers $\beta_{i,j}$ are the {\em graded Betti numbers of the resolution} and every sum $\beta_i:=\sum_j \beta_{i,j}$ is the $i$-th {\em total Betti number of the resolution}. 
If we ask that the above set of generators is minimal, we have a minimal free resolution, that is unique up to changes of coordinates.

The Hilbert Syzygies Theorem guarantees that the minimal free resolution of a homogeneous ideal $I\subset R$ is a {\em finite} exact complex of type $0\rightarrow F_{n} {\buildrel {\delta_n} \over \longrightarrow } \dots {\buildrel {\delta_1} \over \longrightarrow } F_0 {\buildrel {\delta_0} \over \longrightarrow } I \longrightarrow 0$. If  
$F_h=\oplus R^{\beta_{h,j}^I}(-j)$ for every $h=0,\dots,n$, then the maximal integer $h$ such that $F_h\not=0$ is the {\em projective dimension} of $I$.  In this case, every graded Betti number $\beta_{h,j}^I$ is the number of the minimal generators of degree $j$ of the $h$-th module $\mathrm{Syz}^h(I)$ of syzygies of $I$ and is called {\em graded Betti number of~$I$}. Analogously, every sum $\beta_h^I:=\sum_j \beta_{h,j}^I$ is called the $h$-th {\em total Betti number of $I$}.

If $Y=\mathrm{Proj}(R/I)$ is the closed projective scheme defined by a saturated ideal~$I$, then the total (resp.~graded) Betti numbers of $I$ are also called total (resp.~graded) Betti numbers of~$Y$.

For a brief but very clear introduction to {\em Stanley-Reisner} ideals and {\em matroids} we refer to \cite[Subsection 2.1]{Varbaro}. We only recall that the ideals $I_{d,n}\subseteq K[x_1,\dots,x_n]$ generating by all squarefree terms of degree $d$ in $n$ variables are Stanley-Reisner ideals of matroids and that $K[x_1,\dots,x_n]/I_{d,n}$ is a Cohen-Macaulay ring.

\begin{example}\label{ex:primo}
In the polynomial ring $S:=K[x_1,\dots,x_4]$,  the ideal 
$J:=(x_1x_2^2,x_1x_3^2,x_1x_4^3,$ $x_2^2x_3^2,x_2^2x_4^3,
x_3^2x_4^3)$ is generated by all the possible products between two of the terms in $\{x_1,x_2^2,x_3^2,x_4^3\}$. The minimal free resolution of $J$ is
$$0 \to S(-8)^3 \to S(-5)^2 \oplus S(-6)^4 \oplus S(-7)^2 \to S(-3)^2 \oplus S(-4)^2 \oplus S(-5)^2 \to J \to 0$$
with first syzygies generated by 
\vskip 1mm
$[x_2^2, -x_3^2, 0, 0, 0, 0]$, $[0, x_3^2, -x_4^3, 0, 0, 0]$, $[x_1, 0, 0, -x_3^2, 0, 0]$, $[0, 0, 0, x_3^2, -x_4^3, 0]$, 

$[0, x_1, 0, -x_2^2, 0, 0]$, $[0, 0, 0, x_2^2, 0, -x_4^3]$, $[0, 0, x_1, 0, -x_2^2, 0]$, $[0, 0, 0, 0, x_2^2, -x_3^3]$ 
\vskip 1mm
\noindent and second syzygies generated by 
\vskip 1mm
 $[x_1, x_1, -x_2^2, -x_2^2, 0, 0, x_4^3, 0], [0, x_1, 0, -x_2^2, -x_3^2, 0, x_4^3, 0], [0, 0, 0, x_2^2, 0, -x_3^2, 0, x_4^3]$.
\vskip 1mm
\noindent 
In Section \ref{sec:weighted ideals} we will see how this minimal free resolution can be deduced from that of the ideal $I_{2,4}=(x_1x_2,x_1x_3,x_1x_4,x_2x_3,x_2x_4,x_3x_4)$, that is  
$$0 \to S(-4)^3 \to S(-3)^8 \to S(-2)^6 \to I_{2,4} \to 0.$$
The above minimal free resolutions share the same total Betti numbers and, if we consider $I_{2,4}$ with the grading $\deg(x_1)=1$, $\deg(x_2)=2$, $\deg(x_3)=2$, $\deg(x_4)=3$, then the minimal free resolution of $I_{2,4}$ becomes exactly equal to that of $J$ with the standard grading. So, in this case also the graded Betti numbers are the same, and the minimal free resolutions are related as indicated in \cite[Lemma~3.1]{GHMN17}.
\end{example}


\section{Lifting and pseudo-lifting of monomial ideals}
\label{sec:lifting}

For every integer $n\geq 2$ and $t\in \mathbb N$, we consider the variables $x_1,\dots,x_n, u_1,\dots,u_t$ and the polynomial rings $S=K[x_1,\dots,x_n]$ and $R=K[x_1,\dots,x_n,u_1,\dots,u_t]$, over an infinite field $K$. If $t=1$ then we set $u_1:=x_0$. The following definition generalizes the definition of lifting given in \cite{BuchEis,GGR,Ro} in terms of ideals or of $K$-algebras. 

\begin{definition}\cite[Definition 2.3]{MiNa}
Let $E$ be an $S$-module and $F$ and $R$-module. If $u_1,\dots,u_t$ is a $F$-regular sequence and $F/(u_1,\dots,u_t)F\simeq E$, then $F$ is said a {\em $t$-lifting} of $E$ to $R$. If $t=1$, then $F$ is just called a {\em lifting} of $E$.
\end{definition} 

We now focus on a classic lifting procedure that is called $t$-lifting and on a more general procedure that is called pseudo-$t$-lifting. These procedures have been introduced and studied in \cite{MiNa}, and also investigated in \cite{CMR2005}.

\begin{definition} \cite{MiNa}
Given a positive integer $r$, let $A:=(L_{j,i})_{1\leq j\leq n, 1\leq i\leq r}$ be a matrix of type $n\times r$ whose entries $L_{j,i}$ are linear forms in $K[x_j,u_1,\dots,u_t]\subset R$ in which $x_j$ has a non-null coefficient.
We call the matrix $A$ a {\em $t$-lifting matrix}. 
\end{definition}

Given a monomial ideal $J\subset S$, denote by $B_J:=\{\tau_1,\dots,\tau_s\}$ its minimal monomial basis and consider an integer $r$ that is higher than or equal to any non-null exponent of a variable appearing in a term of $B_J$. Then, we can take a $t$-lifting matrix $A$ of type $n\times r$ and associate to $J$  via $A$ the homogeneous ideal $I:= (\tilde \tau_1,\dots,\tilde \tau_s)\subset R$ where, given any term $\tau=x_1^{\alpha_1}\dots x_n^{\alpha_n}\in S$ with $\alpha_i\leq r$ for every $i$, we set
\begin{equation}\label{eq:termine}
\tilde \tau = \prod_{j=1}^{n}\Bigl({\prod_{i=1}^{\alpha_j}L_{j,i}}\Bigr) \in R.
\end{equation}
By \cite[Corollary 2.10]{MiNa}, the module $F=R/I$ is a $t$-lifting of $S/J$ and we will say that $I$ is a {\em $t$-lifting} of $J$ (via the matrix $A$). Moreover, $S/J$ is a Cohen-Macaulay ring if and only if $R/I$ is a Cohen-Macaulay ring (see \cite[Corollary 2.10(ii)]{MiNa}). 

\begin{remark}\label{rem:conditions on lifting} {\rm (\cite[Section 4]{MiNa} and \cite[Proposition 1.7(c)]{CMR2005})}
The following conditions on a $t$-lifting matrix $A:=(L_{j,i})_{1\leq j\leq n, 1\leq i\leq r}$ can be considered:
\begin{itemize}
\item[($\alpha$)] For every $1\leq i_1,\dots, i_n \leq r$, the vector spaces $\langle L_{1,i_1},\dots,L_{n,i_n}\rangle$ are $n$-dimensional or, equivalently, the polynomials $f_j=\prod_{i=1}^r L_{j,i}$, $1\leq j\leq n$, define a complete intersection $(f_1,\dots,f_n)\subset R$ of height $n$.
\item[($\alpha'$)] For every $1\leq i_1,\dots, i_n \leq r$, the vector spaces $\langle L_{1,i_1},\dots,L_{n,i_n}\rangle$ are $n$-dimensional and pairwise distinct or, equivalently, the polynomials $f_j=\prod_{i=1}^r L_{j,i}$, $1\leq j\leq n$, define a {\em reduced} complete intersection $(f_1,\dots,f_n)\subset R$ of height $n$.
\end{itemize}
A $t$-lifting matrix $A$ always satisfies condition ($\alpha$). So, if $I$ is the $t$-lifting of a monomial ideal $J$ via $A$, then $I$ and $J$ have the same height.
\end{remark}

\begin{example}\label{ex:lifting}
The $1$-lifting matrix
$$(L_{j,i})=\left(\begin{array}{ccc} x_1-x_0 & x_1+x_0 & x_1-2x_0 \\ x_2+x_0 & x_2+x_0 & x_2-2x_0 \\ x_3-3x_0 & x_3+3x_0 & x_3-4x_0 \end{array}\right)$$
determines the $1$-lifting $I=((x_1-x_0)(x_1+x_0)(x_3-3x_0), (x_2+x_0)^2 (x_2-2x_0))\subseteq K[x_0,x_1,x_2,x_3]$ of the monomial ideal $J=(x_1^2x_3,x_2^3)\subseteq K[x_1,x_2,x_3]$. Notice that this matrix does not satisfy condition ($\alpha'$) of Remark \ref{rem:conditions on lifting}.
\end{example}

In order to generalize $t$-liftings,  
in \cite{MiNa} matrices with entries in the whole ring $R$ and satisfying condition ($\alpha$) of Remark \ref{rem:conditions on lifting} are introduced.

\begin{definition} \cite{MiNa}
Given a positive integer $r$, let $A:=(L_{j,i})_{1\leq j\leq n, 1\leq i\leq r}$ be a matrix of type $n\times r$ whose entries $L_{j,i}$ are linear forms in the whole ring $R$ satisfying condition ($\alpha$) of Remark \ref{rem:conditions on lifting}. We call the matrix $A$ a {\em pseudo-$t$-lifting matrix}. 
\end{definition}

As for $t$-liftings, given a monomial ideal $J\subset S$ with minimal monomial basis $B_J:=\{\tau_1,\dots,\tau_s\}$, we can take a pseudo-$t$-lifting matrix $A$ and associate to $J$ via $A$ the homogeneous ideal $(\tilde \tau_1,\dots,\tilde \tau_s)\subset R$ (see \eqref{eq:termine}), which is called a {\em pseudo-$t$-lifting} (or {\em distraction}) of $J$ (via $A$). For any ideal $N$ we denote by $\tilde N$ its pseudo-$1$-lifting via the matrix $A$. 

\begin{remark}\label{rem:confronto}
(i) A $t$-lifting is also a pseudo-$t$-lifting, but a pseudo-$t$-lifting is not necessarily a $t$-lifting (e.g.~\cite[Example 1.5]{CMR2005}). 

(ii) Let $A$ be a pseudo-$t$-lifting matrix and $I$ the ideal obtained via $A$ from a monomial ideal $J$. If $A$ also satisfies condition ($\alpha'$) then $\sqrt{I}=I$ (see Remark~\ref{rem:conditions on lifting})  and if, moreover, $J$ is unmixed, then the projective scheme defined by $I$ is equidimensional (see \cite[Propositions 1.14 and 2.5(3)]{CMR2005}).

(iii) The pseudo-$t$-lifting of a monomial ideal
$J=Q_1\cap\ldots\cap Q_s$ is $I = \bar Q_1 \cap\ldots\cap \bar Q_s$ 
where $\bar Q_i$ is the pseudo-$t$-lifting of $Q_i$ (see \cite[Lemma 2.17(ii)]{MiNa}).
So, since a monomial ideal $J$ is the intersection of ideals of type
$(x_{j_1}^{i_1}, \ldots , x_{j_k}^{i_k})$, a pseudo-$t$-lifting $I$ is the intersection of pseudo-$t$-liftings of complete intersections 
$(x_{j_1}^{i_1},\ldots,x_{j_k}^{i_k})$.
\end{remark}

\begin{example}\label{ex:pseudo lifting}
Consider the following pseudo-$1$-lifting matrix
\begin{equation}\label{eq:pseudo}
(L_{j,i})=\left(\begin{array}{lll} 
x_2-x_0+3x_1 & x_1+x_0 & x_3-2x_1\\
x_2-3x_0     & x_2+x_1+3x_0 & x_2-2x_0 \\ 
x_3-3x_0 & x_3+3x_1+x_2 & x_3-3x_2-x_1.
\end{array}\right)
\end{equation}
We checked that this matrix satisfies condition $(\alpha')$. It determines the pseudo-$1$-lifting $I=((x_2-x_0+3x_1)(x_1+x_0)(x_3-3x_0), (x_2-3x_0)(x_2+x_1+3x_0)(x_2-2x_0))\subseteq K[x_0,x_1,x_2,x_3]$ of the monomial ideal $J=(x_1^2x_3,x_2^3)\subseteq K[x_1,x_2,x_3]$.
\end{example}

The following result is crucial for our purpose (see also \cite[Proposition 2.6]{MiNa}).

\begin{proposition} \label{prop:BiCoRo} \cite[Corollary 2.20]{BiCoRo}
Let $\tilde J$ be a pseudo-$t$-lifting of a monomial ideal $J$. The ideals $J$ and $\tilde J$ have the same graded Betti numbers.
\end{proposition}

\begin{remark}\label{rem:dimostrazione indipendente}
In \cite[Example 2.7]{MiNa} it is explicitly observed that the syzygies of a $t$-lifting or pseudo-$t$-lifting of a monomial ideal $J$ do not coincide with the liftings of the syzygies of $J$. This is essentially due to the fact that $\widetilde{\tau\tau'}\not= \tilde \tau \cdot \tilde \tau'$, like it is highlighted in \cite{BiCoRo}. In case of $t$-liftings $I$ of a monomial ideal $J$, $J$ turns out to be the initial ideal of $I$ (for example, see \cite[Theorem 3.2]{BCGR}). Thus, we obtain minimal generators of syzygies of $I$ from minimal generators of syzygies of $J$ by a Gr\"obner basis rewriting procedure, because the Betti numbers coincide.
\end{remark}


\section{Complete grids and complete pseudo-grids of lines}
\label{sec:grids}

In this section, we recall the definitions of complete grid and of complete $m$-fat grid of lines, together with their connection with liftings of monomial ideals that has been studied in the paper \cite{GOR2014} inspired by the investigations described in \cite{CMR2005} and by some examples given in \cite{G2009}. 
With the same perspective, we also introduce the notions of complete pseudo-grid and of complete $m$-fat pseudo-grid of lines.

\begin{definition}\label{def:lattice}\cite[Definition 1]{GO2007} Let $A_0:=\{1\}$, $A_1:=\{a_{11},\dots,a_{1 \ell_1}\}$, $\dots$, $A_n:=\{a_{n 1},\dots,a_{n \ell_n}\}$ be subsets of elements of $K$. Then the set $X \subset \mathbb P^n_K$ consisting of the $\ell_1\ell_2\dots\ell_n$ projective points with coordinates in $A_0\times A_1\times \dots \times A_n$ is called a {\em lattice of type} $(\ell_1,\dots,\ell_n)$. If $\ell_1=\dots=\ell_n=r$, the lattice is said {\em cubic of type $r$}.
\end{definition}

Every lattice $X$ of type $(\ell_1,\dots,\ell_n)$ can be {\em completed} to a cubic lattice of type $r$, $\mathcal X:=\{(1,a_{1 i_1},\dots,a_{n i_n}) \ \vert \ i_1,\dots,i_n \in \{1,\dots,r\}\}$, with $r:=\max\{\ell_1,\dots,\ell_n\}$, by adding $X$ suitable points.

Note that when we consider a projective point $P=(a_0,a_1,\dots,a_n)\in \mathbb P^n_K$ we count the coordinates from $0$ to $n$.

\begin{definition}\label{def:complete grid}\cite[Definition 4]{GO2007}
For every $k\in\{1,\dots,n\}$, denote by $p_{\infty,k}$ the projective point $(0,\dots,1,\dots,0)\in \mathbb P^n_K$ with null coordinates, except the $k$-th coordinate which is equal to $1$, and by $r_{p,k}$ the line through a point $p$ and $p_{\infty,k}$. Given a lattice $X$ of type $(\ell_1,\dots,\ell_n)$, the finite set $Y:=\{r_{p,k}: p\in X, 1\leq k\leq n\}$ 
is called a {\em complete grid of lines of type} $(\ell_1,\dots,\ell_n)$ (with basis the lattice $X$) (for short, {\em complete grid}). 
\end{definition}

\begin{remark}\label{rem:projective closure}
A complete grid of lines $Y\subset \mathbb P^n_K$ is the projective closure of the configuration of all the affine lines through the affine points with coordinates in $A_1\times \dots A_n \subset \mathbb A^n_K$ and parallel to the coordinate axes. Note that the number of lines in a complete grid of type $(\ell_1,\dots,\ell_n)$ is $h=\sum_{i=1}^n \ell_1\dots \hat{\ell_i}\dots\ell_n$.
\end{remark}

\begin{theorem} \label{th:lifting complete grid} 
Let $X$ be a lattice of type $(\ell_1,\dots,\ell_n)$ and $\mathcal X$ a cubic lattice of type~$r$, with $r:=\max\{\ell_1,\dots,\ell_n\}$, containing $X$.
\begin{itemize}
\item[(i)] \cite[Definition 1.1 and Lemma 2.5]{GOR2014} The matrix $\Lambda=(x_j-a_{ji} x_0)_{1\leq j\leq n, 1\leq i\leq r}$ is a $1$-lifting matrix satisfying conditions ($\alpha$) and ($\alpha'$) of Remark \ref{rem:conditions on lifting}.
\item[(ii)] \cite[Theorem 2.6]{GOR2014} The defining ideal $I(Y)$ of the complete grid of lines $Y$ of type $(\ell_1,\dots,\ell_n)$, with basis the lattice $X$, is the $1$-lifting via $\Lambda$ of the monomial ideal $\Omega:=\cap_{k=1}^n (x_1^{\ell_1},\dots,\widehat{x_k^{\ell_k}},\dots,x_n^{\ell_n})$ with minimal set of generators $B_\Omega=\{x_j^{\ell_j} x_k^{\ell_k} \vert 1\leq j<k\leq n\}$.
\end{itemize}
\end{theorem}

\begin{example}\label{ex:grid}
With $n=3$ and $(\ell_1,\ell_2,\ell_3)=(2,3,1)$, the ideal $\Omega$ is generated by the terms $x_1^2x_2^3,x_1^2x_3,x_2^3x_3$. The $1$-lifting matrix satisfying condition $(\alpha')$
$$(L_{j,i})=\left(\begin{array}{ccc} x_1-x_0 & x_1+x_0 & x_1-2x_0 \\ x_2-3x_0 & x_2+3x_0 & x_2-2x_0 \\ x_3-3x_0 & x_3+3x_0 & x_3-4x_0 \end{array}\right)$$
determines the complete grid $Y\subset \mathbb P^3_K$ of type $(2,3,1)$ with defining ideal
$$\begin{array}{lcl}
I(Y)&=& ((x_1-x_0)(x_1+x_0)(x_2-3x_0)(x_2+3x_0)(x_2-2x_0),\\ &&(x_1-x_0)(x_1+x_0)(x_3-3x_0),(x_2-3x_0)(x_2+3x_0)(x_2-2x_0)(x_3-3x_0))\\
&=&  \cap_{i=1}^{2}(L_{1,i},L_{2,1})\bigcap \cap_{i=1}^{2}(L_{1,i},L_{2,2}) \bigcap \cap_{i=1}^{2}(L_{1,i},L_{2,3})\bigcap\cap_{i=1}^{2}(L_{1,i},L_{3,1}) \\
&& \bigcap \cap_{i=1}^{3}(L_{2,i},L_{3,1}).  
\end{array}$$
This complete grid $Y$ consists of $11$ lines over a basis lattice $X$ of $6$ points.
\end{example}

Recall that, given $s$ points $p_1,\dots,p_s$ in a projective space over $K$ and $s$ positive integers $m_1,\dots,m_s$, the closed projective scheme defined by the saturated ideal $\cap_{i=1}^s I(p_i)^{m_i}$ is said a {\em fat point scheme}, and is said a {\em $m$-fat point scheme} if $m=m_1=\dots =m_s$. The following definition provides the analogous notion for complete grids of lines. 

\begin{definition}\label{def:fat complete grid} \cite[Definition 2]{G2014}
Let $Y$ be a complete grid of lines of type $(\ell_1,\dots,\ell_n)$ and $\mathcal P_1,\dots,\mathcal P_h$ the prime ideals defining its $h=\sum_{i=1}^n \ell_1\dots \hat{\ell_i}\dots\ell_n$ lines. 
Given a positive integer $m$, the projective scheme $\mathbb Y$ defined by the saturated ideal $I(\mathbb Y)=\mathcal P_1^{m}\cap \dots \cap\mathcal P_h^{m}$ is called a {\em $m$-fat complete grid of lines} (for short, {\em $m$-fat complete grid}).
\end{definition}

\begin{theorem} \label{th:lifting fat grid} \cite[Theorem 3.5]{GOR2014}
Let $\mathbb Y$ be a $m$-fat complete grid of type $(\ell_1,\dots,\ell_n)$. Consider the linear forms $L_{j,i}=x_j-a_{ji} x_0$ as in Theorem \ref{th:lifting complete grid} and the $1$-lifting matrix $\mathbb A$ whose $j$-th row begins with the sequence $L_{j,1},\dots,L_{j,\ell_j}$ repeated $m$ times, $j\in\{1,\dots,n\}$. Then, the defining ideal $I(\mathbb Y)$ of $\mathbb Y$ is the $1$-lifting of the monomial ideal $\Gamma:=\cap_{k=1}^n (x_1^{\ell_1},\dots,\widehat{x_k^{\ell_k}},\dots,x_n^{\ell_n})^m$ via the matrix $\mathbb A$.  
\end{theorem}

\begin{example}\label{ex:fatgrid}
Let $Y\subseteq \mathbb P^3_K$ be the complete grid of type $(2,3,1)$ of Example~\ref{ex:grid}. Then, the $3$-fat complete grid $\mathbb Y$ supported over $Y$ is the lifting of the ideal $\Gamma$ determined by the matrix
$$(L_{j,i})=\left(\begin{array}{cccccc} x_1-x_0 & x_1+x_0 & x_1-x_0 & x_1+x_0 & x_1-2x_0 &x_1-2x_0 \\ x_2-3x_0 & x_2+3x_0 & x_2-2x_0& x_2-3x_0 & x_2+3x_0 & x_2-2x_0 \\ x_3-3x_0 & x_3+3x_0 & x_3-4x_0 &x_3-3x_0 & x_3+3x_0 & x_3-4x_0 \end{array}\right)$$
and the defining ideal of $\mathbb Y$ is
$$\begin{array}{lcl}I(\mathbb Y)&=& \cap_{i=1}^{2}(L_{1,i},L_{2,1})^3\bigcap \cap_{i=1}^{2}(L_{1,i},L_{2,2})^3 \bigcap \cap_{i=1}^{2}(L_{1,i},L_{2,3})^3  \\ && \bigcap\cap_{i=1}^{2}(L_{1,i},L_{3,1})^3 \bigcap \cap_{i=1}^{3}(L_{2,i},L_{3,1})^3.\end{array}$$
\end{example}

\begin{proposition}\label{prop:acm e risoluzioni}
Let $Y$ be complete grid and $\mathbb Y$ a $m$-fat complete grid. 
\begin{itemize} 
\item[(i)] The graded Betti numbers of $Y$ (resp.~of $\mathbb Y$) coincide with the graded Betti numbers of the ideal $\Omega$ of Theorem \ref{th:lifting complete grid} (resp.~of the ideal $\Gamma$ of Theorem~\ref{th:lifting fat grid}).
\item[(ii)] $Y$ and $\mathbb Y$ are arithmetically Cohen-Macaulay curves. 
\end{itemize}
\end{proposition}

\begin{proof}
Item {\em (i)} follows from Theorem \ref{th:lifting complete grid} (respectively Theorem \ref{th:lifting fat grid}), Remark \ref{rem:confronto}(i) and Proposition \ref{prop:BiCoRo}. Item {\em (ii)} follows from \cite[Corollary 2.10]{MiNa} and Theorems~\ref{th:lifting complete grid} and \ref{th:lifting fat grid}. 
\end{proof}

Recall that we are considering the ideals $\Omega=\cap_{k=1}^n (x_1^{\ell_1},\dots,\widehat{x_k^{\ell_k}},\dots,x_n^{\ell_n})$ and $\Gamma=\cap_{k=1}^n (x_1^{\ell_1},\dots,\widehat{x_k^{\ell_k}},\dots,x_n^{\ell_n})^m$. Note that, as one can easily see, $\Gamma$ is the $m$-th symbolic power of $\Omega$.
We now study pseudo-liftings of $\Omega$ and $\Gamma$.

\begin{definition}\label{def:pseudo}
Let $A=(L_{j,i})_{1\leq j\leq n, 1\leq i\leq r}$ be a pseudo-$1$-lifting matrix satisfying condition~($\alpha'$) of Remark \ref{rem:conditions on lifting} and $\mathbb A$ be the pseudo-$1$-lifting matrix whose $j$-th row begins with the sequence $L_{j,1},\dots,L_{j,\ell_j}$ repeated $m$ times, $j\in\{1,\dots,n\}$.

(1) A {\em complete pseudo-grid} of type $(\ell_1,\dots,\ell_n)$ is the projective scheme defined by the pseudo-$1$-lifting of $\Omega$ via $A$.

(2)  A {\em $m$-fat complete pseudo-grid} of type $(\ell_1,\dots,\ell_n)$ is the projective scheme defined by the pseudo-$1$-lifting of $\Gamma$ via $\mathbb A$.
\end{definition}

\begin{proposition}\label{prop:properties pseudo}
Let $Y^p$ be a complete pseudo-grid and $\mathbb Y^p$ a $m$-fat complete pseudo-grid, as in Definition \ref{def:pseudo}.
\begin{itemize} 
\item[(i)] $Y^p$ is a reduced configuration of projective lines. 
\item[(ii)] The graded Betti numbers of $Y^p$ (respectively of $\mathbb Y^p$) coincide with the graded Betti numbers of the ideal $\Omega$ (respectively of the ideal $\Gamma$).
\item[(iii)] $Y^p$ and $\mathbb Y^p$ are arithmetically Cohen-Macaulay curves. 
\end{itemize}
\end{proposition}

\begin{proof}
For item {\em (i)} recall that $\Omega$ is unmixed and the pseudo-$1$-lifting matrix $A$ satisfies condition $(\alpha'$) (see Remark \ref{rem:confronto}(ii)). Item {\em (ii)} follows from Proposition \ref{prop:BiCoRo}. Item {\em (iii)} holds thanks to item {\em (ii)} and the Auslander-Buchsbaum Formula in the graded case (see \cite[Theorem 4.4.15]{Weibel}). 
\end{proof}

\begin{example} \label{ex:2.fat grids}
Let $\mathbb Y$ be the $3$-fat complete grid of lines of type $(2,3,1)$ of Example~\ref{ex:fatgrid} and $\mathbb Y^p$ be the $3$-fat complete pseudo-grid of lines of type $(2,3,1)$ supported on the complete pseudo-grid determined by the matrix of Example \ref{ex:pseudo lifting}. Both the resolutions of $\mathbb Y$ and $\mathbb Y^p$ have the same graded Betti numbers as the resolution of $\Gamma=(x_1^2,x_2^3)^3\cap (x_1^2,x_3)^3 \cap (x_2^3,x_3)^3$, that is
$0 \to R(-12)^3 \oplus R(-14) \oplus R(-16) \to R(-9)^2 \oplus R(-10) \oplus R(-11) \oplus R(-12) \oplus R(-15)$.
\end{example}

We are denoting by $A=(L_{j,i})$ the $1$-lifting (resp.~$1$-pseudo-lifting) matrix determining the defining ideal $I$ of a complete grid (resp.~pseudo-grid) of lines of type $(\ell_1,\dots,\ell_n)$. Moreover, we are denoting by $\mathbb A$ the $1$-lifting (resp.~pseudo-$1$-lifting) matrix whose $j$-th row begins with the sequence $L_{j,1},\dots,L_{j,\ell_j}$ repeated $m$ times, $j\in\{1,\dots,n\}$. The matrix $\mathbb A$ determines the defining ideal $\mathbf I$ of a $m$-fat complete grid (resp.~pseudo-grid) of lines of type $(\ell_1,\dots,\ell_n)$. 

The following result recalls that $\mathbf I$ is the $m$-th symbolic power of $I$ (see also \cite[Theorem~3.6(1)]{GHMN17}). In general one has $I^{(m)}\not=I^m$.

\begin{proposition}\label{prop:symbolic power}
With the above notation, $\mathbf I$ coincides with the $m$-th symbolic power $I^{(m)}$ of $I$.
\end{proposition}

\begin{proof} By Theorems \ref{th:lifting complete grid} and \ref{th:lifting fat grid} we can describe complete grids and $m$-fat grids as lifting of $\Omega$ and $\Gamma$, respectively, analogously to complete pseudo-grids and $m$-fat pseudo-grids.
In particular, thanks to the properties of associated prime ideals of a lifting (see \cite[Remark 1.11 and Proposition 1.14]{CMR2005}) and to condition $(\alpha')$, we obtain  
$$I= (f_if_j : 1\leq i<j\leq n)= \cap_{k=1}^n (f_1,\dots,\widehat{f_k},\dots,f_n)=$$ 
$$=\bigcap_{k=1}^n \cap_{1\leq i_h\leq \ell_h, 1\leq h \leq n}  (L_{1,i_1},L_{2,i_2},\dots,\widehat{L_{k,i_k}},\dots,L_{n,i_n}).$$
So, the primes associated to $I$ are the ideals $(L_{1,i_1},L_{2,i_2},\dots,\widehat{L_{k,i_k}},\dots,L_{n,i_n})$, which are generated by regular sequences. The symbolic power of an ideal generated by a regular sequence coincides with the power of the ideal \cite[Lemma 5, Appendix 6]{ZarSam2} and, taking into account that localization commutes with both intersections and finite products, we obtain 
$$I^{(m)}= \cap_{{\mathfrak p}\in \mathrm{Ass}(R/I)}(I^m R_{\mathfrak p}\cap R)= \cap_{{\mathfrak p}\in \mathrm{Ass}(R/I)}(\mathfrak p^m R_{\mathfrak p}\cap R) =$$
$$=\bigcap_{k=1}^n  \cap_{1\leq i_h\leq \ell_h, 1\leq h \leq n}  (L_{1,i_1},L_{2,i_2},\dots,\widehat{L_{k,i_k}},\dots,L_{n,i_n})^m=$$
$$= \cap_{k=1}^n (f_1,\dots,\widehat{f_k},\dots,f_n)^{m}=\mathbf I,$$ 
also using Definition \ref{def:pseudo}(2), the shape of the generators of the ideals $(x_1^{\ell_1},\dots,\widehat{x_k^{\ell_k}},\dots,x_n^{\ell_n})^m$ and Remark \ref{rem:conditions on lifting}(iii).
\end{proof}

Thanks to the relation between liftings and complete grids, it is possible to recognize that complete grids and pseudo-grids of lines are particular {\em $\lambda$-configurations} of codimension $n-1$. Let us briefly recall what a $\lambda$-configuration is.
Let $f_1,\dots,f_s\in R$ be homogeneous forms with positive degrees $d_1,\dots,d_s$ such that, denoted by $F_1,\dots,F_s$ the hypersurfaces they define in $\mathbb P^n_K$, for any $1\leq c\leq n$ the intersection of any $c+1$ of these hypersurfaces has codimension $c+1$. 
Setting $\lambda:=[d_1,\dots,d_s]$,  
a hypersurface configuration of codimension $c$ is the union $V_{\lambda,c}$ of all the codimension~$c$ complete intersection subschemes obtained by intersecting $c$ of the hypersurfaces (see \cite{GHMN17} and also \cite[Definition 2.1]{PS15}). The minimal generators of $I(V_{\lambda,c})$ are all the products of $s-c+1$ of the forms $f_1,\dots,f_s$ (see \cite[Proposition 2.3(4)]{GHMN17}). 
The hypersurface configurations for which $d_1=\dots=d_n=1$ include the so-called {\em codimension $c$ (linear) star-configurations} (see also \cite{GHM13}).
If $\lambda\not=[1,\dots,1]$ then a hypersurface configuration is also called a {\em $\lambda$-configuration}.

In order to see that complete grids and pseudo-grids are hypersurface configurations, it is enough to take $\lambda=[\ell_1,\dots,\ell_n]$ and the forms $f_1=\prod_{i=1}^{\ell_1} L_{1,i},\dots,f_n=\prod_{i=1}^{\ell_n} L_{n,i}$, where $A=(L_{j,i})$ is the $1$-lifting (resp.~$1$-pseudo-lifting) matrix satisfying condition~$(\alpha')$ and determining the ideal $I$ of the given complete grid (or  pseudo-grid). In fact, $I$ is the ideal obtained replacing every variable $x_i$ by the polynomial $f_i$ in the generators of the Stanley-Reisner ideal $I_{2,n}$.


\section{Weighted ideals and Betti numbers of $\Omega$ and $\Gamma$}\label{sec:weighted ideals}

The shape of the monomial ideals $\Omega$ and $\Gamma$ suggests the following definition. Recall that $B_J$ denotes the minimal monomial generating set of a monomial ideal $J$.

\begin{definition}\label{def:weighted}
Given $n$ positive integers $\ell_1,\dots,\ell_n$, a monomial ideal $J\subset S$ is said {\em weighted} (by ($\ell_1,\dots,\ell_n$)) if $B_J$ consists of terms of type $x_1^{\alpha_1\ell_1} \dots x_n^{\alpha_n\ell_n}$.
\end{definition}
 
Take a set of new variables $y_1,\dots,y_n$ and the polynomial ring $P:=K[y_1,\dots,y_n]$. Given $n$ positive integers $\ell_1,\dots,\ell_n$, we consider the following ring monomorphism:
\begin{equation}\label{eq:monomorfismo}
\Phi: y_i \in K[y_1,\dots,y_n] \rightarrow x_i^{\ell_i}\in K[x_1,\dots,x_n].
\end{equation}
This is a particular case of a flat map studied by Hartshorne in \cite{Hart1966}. Here we focus on this particular case.

For convenience, we will denote by $\Phi$ also every map $y_i e_k \in K[y_1,\dots,y_n]^t \rightarrow x_i^{\ell_i} e_k\in K[x_1,\dots,x_n]^t$, where $e_k$ is the $k$-th generator of the canonical basis of $K[y_1,\dots,y_n]^t$ (respectively, of $K[x_1,\dots,x_n]^t$). 

Given a weighted ideal $J\subset S$, there is always a monomial ideal $\overline{J}\subset K[y_1,\dots,y_n]$ such that $\Phi(B_{\overline J})=B_J$, i.e.~$\Phi(\overline J)=J$. In fact, it is obvious that every monomial ideal is weighted at least by $(1,\dots,1)$, but this case is banal. 

The ring monomorphism $\Phi$ gives the following crucial connection between the syzygies of $\overline J$ and those of $J$.

\begin{theorem}\label{th:betti weighted}
For every $h\in\{0,\dots,n\}$, the map $\Phi$ provides a bijection between a set of minimal generators of the $h$-th module $\overline F_h$ of syzygies of $\overline J$ and a set of minimal generators of the $h$-th module $F_h$ of syzygies of the weighted ideal $J~=~\Phi(\overline J)K[x_1,\dots,x_n]$. In particular, the total Betti numbers of $\overline J$ coincide with the total Betti numbers of $J$.
\end{theorem}

\begin{proof}
We proceed by induction on the index $h$ of the modules $F_h$ and first we observe that for $h=0$ the statement holds by definition of the ideals $\overline J$ e $J$. 

Now, assume that the statement is true for $h-1$. Thus, minimal $(h-1)$-syzygies of $J$ are vectors with components equal to $K$-linear combinations of terms of type $x_1^{a_1\ell_1}\dots x_n^{a_n\ell_n}$. Multiplications, divisions and computations of least common multiples between terms of this type are still terms of the same type. Hence, following standard constructions of syzygies that are based on Gr\"obner bases rewriting procedures, we can compute 
minimal $h$-th syzygies of $J$ that are of the same type of the $(h-1)$-th syzygies, i.e.~vectors with components equal to $K$-linear combinations of terms of type $x_1^{a_1\ell_1}\dots x_n^{a_n\ell_n}$. So, every minimal $h$-th syzygy of $J$ is the image of a minimal $h$-th syzygy of $\overline J$ by the map $\Phi$. 

Conversely, by construction, a $h$-th syzygy of $\overline J$ is transformed by the map $\Phi$ in a $h$-th syzygy of $J$. The injectivity of the map $\Phi$ preserves the minimality of the generators.
\end{proof}

\begin{remark}
The lattice $X$ of Theorem \ref{th:lifting complete grid} is a $1$-lifting of the ideal $\mathfrak b=(x_1^{\ell_1},\dots,x_n^{\ell_n})$ \cite[Theorem 2.6]{GOR2014}. Hence, we have $\overline{\mathfrak b}=(y_1,\dots,y_n)$ and the total Betti numbers of a polynomial complete intersection of height $n$ coincide with the total Betti numbers of the ideal $(y_1,\dots,y_n)$.
\end{remark}

If the map $\Phi$ is determined by pairwise equal integers $\ell_1=\dots=\ell_n=\ell$, then $\Phi$ behaves as a homogeneous map over the syzygies of $\overline J$ and the graded Betti numbers are preserved up to a multiplication of the shifts by $\ell$, as it is easy to see. The particular case $\ell_1=\dots=\ell_n=1$ corresponds to the fact that $J$ and $\overline J$ coincide up to the name of the variables of the polynomial rings, in our context. The following example shows that, when the integers $(\ell_1,\dots,\ell_n)$ are not equal each other, the graded Betti numbers of $\overline J$ and $J= \Phi( \overline J)$ are preserved up to {\em additional} degree shifts corresponding to the integers $\ell_i$, according to \cite[Lemma~3.1]{GHMN17}. However, assigning the ring $K[y_1,\dots,y_n]$ the different grading with $deg(y_i)=\ell_i$, the homomorphism $\Phi$ becomes a homomorphism of graded rings and the graded Betti numbers are exactly the same, according to \cite[Lemma~3.1]{GHMN17} too.

\begin{example}\label{ex:unmixed}
Consider the ideal $\overline J=(y_1,y_2^2) \cap (y_1^2,y_2,y_3) = (y_1y_3, y_1y_2, y_2^2, y_1^2)\subset P=K[y_1,y_2,y_3]$. The minimal free resolution of $\overline J$ is:
$$0\rightarrow P(-4) \rightarrow P^4(-3) \rightarrow P^4(-2)\rightarrow \overline J \rightarrow 0,$$
with generators of the first module of syzygies $v_1=[y_2, -y_3, 0, 0]$, $v_2=[y_1, 0, 0, -y_3]$, $v_3=[0, y_2, -y_1, 0]$, $v_4=[0, y_1, 0, -y_2]$ and of the second module of syzygies $w=[y_1,-y_2,0,y_3]$. Then, for every $\ell_1,\ell_2,\ell_3$, the minimal free resolution of the weighted ideal $J=(x_1^{\ell_1},x_2^{2\ell_2}) \cap (x_1^{2\ell_1},x_2^{\ell_2},x_3^{\ell_3}) = (x_1^{\ell_1}x_3^{\ell_3}, x_1^{\ell_1}x_2^{\ell_2}, x_2^{2\ell_2}, x_1^{2\ell_1})\subset S=K[x_1,x_2,x_3]$ is:
$$0\rightarrow S(-2\ell_1-\ell_2-\ell_3) \rightarrow S(-\ell_1-\ell_2-\ell_3)\oplus S(-2\ell_1-\ell_3) \oplus S(-\ell_1-2\ell_2) \oplus S(-2\ell_1-\ell_2) \rightarrow$$
$$\rightarrow S(-\ell_1-\ell_3)\oplus S(-\ell_1-\ell_2)\oplus S(-2\ell_2) \oplus S(-2\ell_1) \rightarrow J \rightarrow 0,$$
with first module of syzygies generated by $\Phi(v_1)=[x_2^{\ell_2},-x_3^{\ell_3},0,0]$, $\Phi(v_2)=[x_1^{\ell_1},0,0,-x_3^{\ell_3}]$, $\Phi(v_3)=[0,x_2^{\ell_2},-x_1^{\ell_1},0]$, $\Phi(v_4)=[0,x_1^{\ell_1},0,-x_2^{\ell_2}]$ and second module of syzygies generated by $\Phi(w)=[x_1^{\ell_1},x_2^{\ell_2},0,x_3^{\ell_3}]$. If we assign $P$ the non standard grading $\deg(y_i)=\ell_i$, then the minimal free resolution of $\overline J$ is just
$$0\rightarrow P(-2\ell_1-\ell_2-\ell_3) \rightarrow P(-\ell_1-\ell_2-\ell_3)\oplus P(-2\ell_1-\ell_3) \oplus P(-\ell_1-2\ell_2) \oplus P(-2\ell_1-\ell_2) \rightarrow$$
$$\rightarrow P(-\ell_1-\ell_3)\oplus P(-\ell_1-\ell_2)\oplus P(-2\ell_2) \oplus P(-2\ell_1) \rightarrow \overline J \rightarrow 0.$$
\end{example}

Among the results in the following proposition, {\em (i), (ii)} can be obtained from~\cite{Hart1966} and {\em (iii)} is analogous to \cite[Theorem~3.3(1)]{GHMN17}, which is differently proved for Stanley-Reisner ideals of matroids.

\begin{proposition}\label{prop:osservazioni weighted} 
Let $J\subset S$ be a monomial weighted ideal and $\overline J\subset P$ be a monomial ideal such that $\Phi(B_{\overline J})=B_J$.
\begin{itemize}
\item[(i)] $J$ is unmixed $\Leftrightarrow$ $\overline J$ is unmixed. 
\item[(ii)] $S/J$ and $P/\overline J$ have the same Krull dimension.
\item[(iii)] $S/J$ is Cohen-Macaulay $\Leftrightarrow$ $P/\overline J$ is Cohen-Macaulay.  
\end{itemize}
\end{proposition}

\begin{proof}
For items {\em (i), (ii)}, it is enough to observe that a primary radical decomposition of $\overline J$ is preserved by $\Phi$ because $\Phi$ is a ring monomorphism. 

For item {\em (iii)}, thanks to Proposition \ref{prop:osservazioni weighted}, we know that $J$ is unmixed if and only if $\overline J$ is unmixed and that $K[x_1,\dots,x_n]/J$ and $K[y_1,\dots,y_n]/\overline J$ have the same Krull-dimension. From Theorem \ref{th:betti weighted} we obtain that $K[x_1,\dots,x_n]/J$ and $K[y_1,\dots,y_n]/\overline J$ have the same projective dimension, that proves the statement thanks to the Auslander-Buchsbaum Formula in the graded case (see \cite[Theorem 4.4.15]{Weibel} and for the graded case \cite[Exercise 19.8]{Ei}).  
\end{proof}

\begin{example}\label{ex:unmixed2}
Consider the non-unmixed ideal $\overline J=(y_1,y_2^2) \cap (y_1^2,y_2,y_3) = (y_1y_3, y_1y_2, y_2^2, y_1^2)$ in $P=K[y_1,y_2,y_3]$ of Example \ref{ex:unmixed}. The Krull dimension of $P/\overline J$ is $1$. Then, for every positive integers $\ell_1,\ell_2,\ell_3$, the weighted ideal $J=(x_1^{\ell_1},x_2^{2\ell_2}) \cap (x_1^{2\ell_1},x_2^{\ell_2},x_3^{\ell_3}) = (x_1^{\ell_1}x_3^{\ell_3}, x_1^{\ell_1}x_2^{\ell_2},$ $x_2^{2\ell_2}, x_1^{2\ell_1})\subset S=K[x_1,x_2,x_3]$ is non-unmixed and the Krull dimension of $S/J$ is $1$.
\end{example}

The unmixed monomial ideal $\Omega=\cap_{k=1}^n (x_1^{\ell_1},\dots,\widehat{x_k^{\ell_k}},\dots,x_n^{\ell_n})$ is generated by $B_\Omega=\{x_i^{\ell_i} x_j^{\ell_j} : 1\leq i<j \leq n\}$. So, $\Omega$ is weighted by ($\ell_1,\dots,\ell_n$) and an ideal $\overline{\Omega}\subset K[y_1,\dots,y_n]$ satisfying the equality $B_\Omega=\Phi(B_{\overline{\Omega}})$ is generated by $B_{\overline{\Omega}}=\{y_i y_j : 1\leq i<j\leq  n\}$.
This ideal $\overline\Omega$ is the Stanley-Reisner ideal of a matroid, more precisely is the ideal $I_{2,n}$ already introduced at the end of Section \ref{sec:preliminaries}. For this ideal, one finds a relation between the result of Theorem~\ref{th:betti weighted} and that of \cite[Theorem 3.3(3)]{GHMN17}. Nevertheless, the investigation of \cite[Theorem 3.3(3)]{GHMN17} is based on \cite[Lemma 3.1]{GHMN17}, which is in part contradicted by Example \ref{ex:unmixed}, as already observed.

The ideal $\overline{\Omega}=\cap_{k=1}^n(y_1,\dots,\hat y_k,\dots,y_n)$ is the saturated ideal of the scheme of the $n$ projective points $(1,0,\dots,0),\dots,(0,\dots,0,1)$ in $\mathbb P^{n-1}_K$. Hence, the graded Betti numbers of $\overline{\Omega}$ coincide with the graded Betti numbers of the defining ideal of any set of $n$ general points in $\mathbb P^{n-1}_K$, because a linear and invertible change of coordinates preserves graded Betti numbers. We mean that a set of $n$ points in $\mathbb P^{n-1}_K$ is general if it belongs to the open subset of $\mathbb P^{n-1}_K\times \dots \times \mathbb P^{n-1}_K$ ($n$ times) made of all subsets of $n$ points of $\mathbb P^{n-1}_K$ not contained in any linear projective subvariety of dimension lower than $n-1$ (see also \cite{F2001}).

Next Lemma \ref{lemma:risoluzione punti} is a particular case of \cite[Theorem~2.1]{Gal16} and \cite[Theorem~4.1]{ELSW18} (see also \cite[Remark 2.11]{GHM13}, \cite[Theorem 3.4]{PS15}).
We now give an independent geometrical proof of this  particular case.

\begin{lemma}\label{lemma:risoluzione punti}
The minimal free resolution of the ideal $\overline{\Omega}\subset P$ is of type
$$0 \rightarrow P^{\overline\beta_{n-2}}(-n) \rightarrow \cdots \rightarrow P^{\overline\beta_2}(-4) \rightarrow P^{\overline\beta_1}(-3) \rightarrow P^{\overline\beta_0}(-2) \rightarrow \overline{\Omega} \rightarrow 0,$$
where, for every $h\in\{0,\dots,n-2\}$, $\displaystyle\overline\beta_h=(h+1)\binom{n}{h+2}$ coincides with the $h$-th Betti number of the rational normal curve in $\mathbb P^n_K$.
\end{lemma}

\begin{proof}
Like we have just observed, the ideal $\overline{\Omega}$ defines a scheme of $n$ projective general points in $\mathbb P^{n-1}_K$. Hence, $\overline \Omega$ has the same Betti numbers of any scheme of $n$ projective general points in $\mathbb P^{n-1}_K$, like a general hyperplane section $Z$ of the rational normal curve $C\subset \mathbb P^n_K$. The $h$-th Betti number of the rational normal curve is $(h+1)\binom{n}{h+2}$, for every $h\in\{0,\dots,n-2\}$ (e.g.~\cite[Corollary 6.2]{Eisenbud}), and coincides with that of $Z$ because $C$ is arithmetically Cohen-Macaulay (for example, see \cite[Theorem 1.3.6]{Mi}). Hence, we can conclude. 
\end{proof}

\begin{proposition}\label{prop:betti grids}
For every $h\in \{0,\dots,n-2\}$, let $\beta^\Omega_h$ be the $h$-th total Betti number of the ideal $\Omega$ and $\beta^Y_h$ the $h$-th total Betti number of a complete grid of lines $Y$ in $\mathbb P^n_K$. Then, $\beta^\Omega_h=\beta^Y_h=(h+1)\binom{n}{h+2}$.
\end{proposition}

\begin{proof}
The thesis follows from Lemma \ref{lemma:risoluzione punti}, Theorem \ref{th:betti weighted} and Proposition \ref{prop:acm e risoluzioni}.
\end{proof}

Consider the monomial ideal $\Gamma=\cap_{k=1}^n (x_1^{\ell_1},\dots,\widehat{x_k^{\ell_k}},\dots,x_n^{\ell_n})^m \subset S$. For $n=3$ the minimal monomial basis of $\Gamma$ is explicitly described in \cite{GOR2014}. 
We now find the minimal monomial basis of $\Gamma$ for every $n$, observing that $\Gamma$ is a weighted ideal as well.  
The ideal $\overline\Gamma=\cap_{k=1}^n (y_1,\dots,\hat y_k,\dots,y_n)^m \subset K[y_1,\dots,y_n]$ is the saturated ideal of the $m$-fat point scheme that is supported on the $n$ projective points $(1,0,\dots,0),\dots,(0,\dots,0,1)$ in $\mathbb P^{n-1}_K$. Hence, in particular, $\overline\Gamma$ is the $m$-th symbolic power of $\overline\Omega$. The minimal monomial basis of $\overline\Gamma$ is well-know and described in \cite[Theorem 2.4]{F2001} for points in general position, which can be assumed to be the above points because they have the same Hilbert function, Betti numbers and resolution. 

Analogously to the case of $\overline\Omega$ and $\Omega$, one can see that $\overline\Gamma$ and  $\Gamma$ generally do not share the same graded Betti numbers. 

\begin{lemma}\label{lemma:generatori fat grid}
The ideal $\Gamma\subset S$ is weighted by ($\ell_1,\dots,\ell_n$) with minimal monomial basis 
$$\begin{array}{ll}B_{\Gamma}=\Phi( B_{\overline\Gamma})=\cup_{t=1}^m \Bigl\{ x_1^{\ell_1 b_1}\dots x_n^{\ell_n b_n} \in S_{m+t} \ \vert &\sum_i b_i=m+t;  b_i\leq t \ \forall i=1,\dots,n; \\
  &\exists \ 1\leq u<v\leq n \text{ with } b_u=b_v=t \ \ \Bigr\}.
    \end{array}$$
\end{lemma}

\begin{proof} Note that
$\Phi(\overline\Gamma)= \Phi(\cap_{k=1}^n (y_1,\dots,\hat y_k,\dots,y_n)^m)= \cap_{k=1}^n \Phi(y_1,\dots,\hat y_k,\dots,y_n)^m = \Gamma$,  
because $\Phi$ is an injective ring homomorphism. Then, we conclude thanks to the explicit description of $B_{\overline\Gamma}$ that is given in \cite[Theorem 2.4]{F2001}.  
\end{proof}

\begin{proposition}\label{prop:betti fat grids}
For every $h\in \{0,\dots,n-2\}$, the $h$-th total Betti number $\beta^{\Gamma}_h$ of the ideal $\Gamma$ and the $h$-th total Betti number $\beta^{\mathbb Y}_h$ of a $m$-fat complete grid of lines $\mathbb Y$ in $\mathbb P^n_K$ coincide with the $h$-th total Betti number of $n$ general $m$-fat points in $\mathbb P^{n-1}_K$. 

If $\mathbb Y$ is a $m$-fat complete grid of type $(1,\dots,1)$, then also the graded Betti numbers coincide with those of $n$ general $m$-fat points in $\mathbb P^{n-1}_K$.
\end{proposition}

\begin{proof}
From Proposition \ref{prop:acm e risoluzioni}, the graded Betti numbers of $\mathbb Y$ and of the ideal $\Gamma$ coincide. In particular, for every $h\in \{0,\dots,n-2\}$, $\beta^{\mathbb Y}_h$ is equal to the $h$-th total Betti number of $\Gamma$. On the other hand, the $h$-th total Betti number of the ideal $\Gamma$ coincides with the $h$-th total  Betti number of the ideal $\overline{\Gamma}$, thanks to Lemma \ref{lemma:generatori fat grid} and Theorem~\ref{th:betti weighted}. We now can conclude because $\overline{\Gamma}$ is the defining ideal of the $m$-fat point scheme that is supported on the $n$ projective points $(1,0,\dots,0),\dots,(0,\dots,0,1)$ in $\mathbb P^{n-1}_K$ and a linear and invertible change of coordinates preserves Betti numbers, as we have already recalled. For the last statement, it is enough to observe that in this case $\overline{\Gamma}=\Gamma$, up to the name of the variables.
\end{proof}

\begin{example}\label{ex:dalle rette ai punti}
From \cite{G2014} the minimal free resolution of a $m$-fat complete grid $\mathbb Y$ of type $(1,1,1)$ in $\mathbb P^3_K$ is \  $0\rightarrow F_1 \rightarrow F_0 \rightarrow 0$ with:

$F_0= \bigoplus_{q=0}^{\frac{m-1}{2}} R^3(-q-2(m-q)),$

$F_1= \bigoplus_{q=1}^{\frac{m-1}{2}} R^3(-q-2(m-q+1))\bigoplus R^2(-3(\frac{m+1}{2}))$

\noindent if $m$ is odd, and with

$F_0= \bigoplus_{q=0}^{\frac{m-2}{2}} R^3(-q-2(m-q)) \bigoplus R(-3\frac{m}{2}),$

$F_1= \bigoplus_{q=1}^{\frac{m-2}{2}} R^3(-q-2(m-q+1)) \bigoplus R^3(-2-3(\frac{m}{2}))$

\noindent if $m$ is even. Then, thanks to Proposition \ref{prop:betti fat grids}, the minimal free resolution of a $m$-fat point scheme $Z$ in $\mathbb P^2_K$ supported on $3$ general points is (see \cite{Catalisano} for a more general result over an algebraically closed field): 
\begin{itemize}
\item $0 \rightarrow \oplus_{i=1}^{\frac{m-1}{2}}P^3(-\alpha-1-i) \oplus P^2(-\alpha-1) \rightarrow \oplus_{i=0}^{\frac{m-1}{2}} P^3 (-\alpha-i) \rightarrow 0$, 

\noindent where $\alpha=m+\frac{m+1}{2}$, if $m$ is odd;
\item $0 \rightarrow \oplus_{i=1}^{\frac{m}{2}}P^3(-\alpha-i-1) \rightarrow \oplus_{i=1}^{\frac{m}{2}} P^3 (-\alpha-i) \oplus P(-\alpha) \rightarrow 0$, 

\noindent where $\alpha=m+\frac{m}{2}$, if $m$ is even.
\end{itemize}
\end{example}

We can find descriptions of the graded Betti numbers of the $m$-th symbolic power ${I_{d,n}}^{(m)}$ in several cases (see \cite[Corollaries 4.4, 4.5, 4.6]{BAG+19}, \cite[Theorem 3.2]{GHM13}, \cite[Proposition 3.2]{Francisco2005}). Hence, in the cases for which $d=2$, we also have the graded Betti numbers of the $m$-fat point scheme supported on $n$ general points in $\mathbb P^{n-1}_K$, but only the total Betti numbers of fat grids and pseudo-grids. See also \cite[Theorem~3.6(2)]{GHMN17}.


\section{Minimal free resolution of a complete grid and pseudo-grid}\label{sec:mapping cone}

In this section, referring to \cite{CEv,HT,PS}, we apply an iterated mapping cone construction in order to obtain a free resolution of a complete grid or pseudo-grid that is minimal thanks to the results of Section \ref{sec:weighted ideals}. First, we recall the definition of mapping cone and the idea that supports the iterated procedure.

\begin{definition} \cite[Appendix A3]{Ei} 
If $ \alpha :  {\bf E} \ra {\bf F}$ is a map of complexes and we write $\delta$ and $\varphi$, respectively, for the differentials of ${\bf E}$  and $ {\bf F}$, then the {\em mapping cone} of $\alpha$ is the complex ${\bf G}$ with \, $G_i := E_{i-1} \oplus F_i$ and differentials \, $\gamma _{i} : G_{i} = E_{i-1} \oplus F_{i} \ra G_{i-1}= E_{i-2} \oplus F_{i-1}$ defined by \, $\gamma_{i}(a,b) := (- \delta_{i-1}(a), \, \alpha_{i-1}(a) + \varphi_{i}(b))$.
\end{definition}

Let $I \subset R=K[x_0,x_1,\dots,x_n]$ be a homogeneous ideal generated by $f_1, \ldots, f_m$. For every $c\in\{1,\dots,m\}$ set $I_{c} := (f_{c}, \ldots, f_m)$ and recall that there are the following short exact sequences with homomorphisms of degree $0$:
$$ 0 \longrightarrow \frac{R}{(I_{c+1}: (f_{c}))} \, (- \, deg(f_{c}))  \stackrel{\cdot f_{c}}{\longrightarrow}  \frac{R}{I_{c+1}} \longrightarrow \frac{R} {(f_{c})+I_{c+1}} \longrightarrow 0.$$
If a free $R$-resolution ${\bf E} $ of $ R/ (I_{c+1}: (f_{c}))$ and a free $R$-resolution ${\bf F}$  of $ R/ I_{c+1}$ are known, then we can find  a complex map $ \alpha :  {\bf E}  \, (- \, deg(f_{c})) \ra {\bf F}$ which is a lifting of the multiplication by $f_{c}$ and whose comparison maps $\alpha_i$ are of degree $0$. We now recall how to find such a map $\alpha$:
$$\begin{array}{ccccccccl}
\cdots&\ra& E_2 (- deg(f_{c}))&  \stackrel{\delta_2}{\ra} & E_1 (-  deg(f_{c})) &  \stackrel{\delta_1}{\ra} & E_0 (-  deg(f_{c}))&  \stackrel{\delta_0}{\ra} & R (- deg(f_{c})) \\
 & & \downarrow  \alpha_2 & &  \downarrow  \alpha_1 & &  \downarrow  \alpha_0 & &  \downarrow  \cdot \, f_{c} \\
\cdots&\ra& F_2 &  \stackrel{\varphi_2}{\ra} & F_1  &  \stackrel{\varphi_1}{\ra} & F_0 & \stackrel{\varphi_0}{\ra} & R 
\end{array}$$
Let $\bar a$ be an element of a minimal system of generators of $E_0 (-  deg(f_{c}))$, we have $\delta_0(\bar a) \in (I_{c+1} : (f_{c}))$ and $ \delta_0(\bar a) f_{c} \in I_{c+1}$. Thus there exists $b \in F_0$ such that $\varphi _0 (b) = \delta_0(\bar a) f_{c}$ and we can define $\alpha_0(\bar a)=b$.

Now, let $a$ be an element of a minimal system of generators of $E_1 (-  deg(f_{c}))$, we have $\delta_0 \circ \delta_1 (a) = 0$ in $R$, then $ \varphi_0 \circ \alpha_0 \circ \delta_1 (a) =0$ and $ \alpha_0 \circ \delta_1 (a) \in Ker \varphi_0 = Im \varphi_1$. So there exists $b \in F_1$ such that $\varphi_1 (b) =  \alpha_0 \circ \delta_1 (a)$ and we define $\alpha_1(a)=b$. 

And so on, we continue this construction obtaining a free resolution of  $R/I_{c}$,  where $I_{c} =  (f_{c})+I_{c+1}$, as a mapping cone of the complex map $\alpha$:
$$\cdots \ra  E_1 (-  deg(f_{c})) \oplus F_2  \stackrel{\gamma_2}{\ra}  E_0 (-  deg(f_{c})) \oplus F_1  \stackrel{\gamma_1}{\ra}  R (- deg(f_{c})) \oplus F_0  \stackrel{\gamma_0}{\ra} R \ra \frac{R}{(f_{c})+I_{c+1}}.$$

Note that the map $\gamma_0$ is given by the generators $f_{c}, \ldots, f_n$ of the ideal $I_{c}$. Moreover, for every $i \geq 1$, the matrix $M_{\gamma_i}$ associated to the map $\gamma _i$ (with respect to the canonical bases) is 
\begin{equation}\label{eq:matrice sizigie} 
M_{\gamma_i} = \left(
\begin{array}{cc}
- \, M_{\delta_{i-1}} & 0 \\
M_{\alpha_{i-1}} & M_{\varphi_{i}}\\
\end{array} \right).
\end{equation}
Such a resolution can be non-minimal. If $Im ( \alpha _i) \subseteq \gm \, F_i$, that is the entries of the matrix $M_{\alpha_i}$ belong to the irrelevant maximal ideal $\gm$, then the free resolution obtained as mapping cone of $\alpha$ is minimal. However,  this construction yields an inductive procedure to compute a resolution of $R/I$.

\begin{remark}\rm
Suppose that for the above minimal free resolutions ${\bf E}, {\bf F}$ we have \, $E_i = \oplus _j R(- \epsilon _{ij})$, $F_i = \oplus _j R(- \lambda _{ij})$.
If $ \epsilon _{ij} + deg(f_{c}) > \lambda _{ih}$ for all $i,j,h$, then the above free resolution mapping cone of $\alpha$ is minimal (see \cite[Appendix A3]{Ei}). In fact $\epsilon _{ij} + deg(f_{c})$ is the degree of the $j$-th element $e_j$ of the canonical basis of $E_i(-deg(f_{c}))$, $ \alpha_i(e_j)= ( \ldots, \phi_{hj}, \ldots)$ and every $ \phi_{hj}$ is of degree  $ \epsilon _{ij} + deg(f_{c})$ in $R(-\lambda _{ih})$. Note that $ \phi_{hj} \in \gm =R_1$ if and only if  $ \epsilon _{ij} + deg(f_{c}) > \lambda_{ih}$. In literature there are also other criteria for the minimality of a resolution that is obtained by mapping cone (e.g.~\cite[Lemma 2.3]{Francisco2005}). However, up to our knowledge, they do not apply to the case we are considering.
\end{remark}

Now, we go back to our ideal $\Omega= ( x_1^{\ell_1} x_2^{\ell_2}, x_1^{\ell_1}x_3^{\ell_3}, \ldots , x_1^{\ell_1}x_n^{\ell_n}, \ldots , x_{n-1}^{\ell_{n-1}}x_n^{\ell_n}) \subset S$. Let $f_1,\dots,f_n$ be a regular sequence of homogeneous forms in $R$ with degrees $\ell_1,\dots,\ell_n$ and let $\tilde\Omega:=(f_1 f_2, f_1 f_3, \ldots , f_1 f_n, \dots, f_{n-1} f_n)\subseteq R$.

For every $1 \leq k \leq n$, we consider the ideal $\tilde\Omega_{[k]}$ generated by the polynomials  $f_i f_j$ with $1\leq k \leq i < j \leq n$. Observe that $\tilde\Omega= \tilde\Omega_{[1]}$ and moreover
$$\tilde\Omega = ( f_1 f_2, f_1 f_3, \dots, f_1 f_n) + \tilde\Omega_{[2]}.$$
Since $f_1,\dots,f_n$ is a regular sequence of homogeneous forms, we obtain:
$$(\tilde\Omega_{[2]} : ( f_1f_n)) =  (f_2, \dots , f_{n-1}) =: \gb_n,$$
$$((f_1 f_n) + \tilde\Omega_{[2]}) : (f_1 f_{n-1})) = (f_2, \dots,\widehat{f_{n-1}}, f_n) =: \gb_{n-1},$$
and, for every $3 \leq k \leq n$,
$$(( f_1f_k, \dots, f_1f_n) + \tilde\Omega_{[2]} ) : ( f_1 f_{k-1})) = (f_2, \dots ,\widehat{ f_{k-1}}, \dots, f_n)  =: \gb_{k-1}.$$
An ideal of type $\gb := (f_{1}, f_{2}, \dots, f_{n-2})$ is a complete intersection and, hence, its minimal free resolution is
$0 \ra E_{n-3} \ra E_{n-4} \ra \cdots \ra E_1 \ra E_0 \ra \gb \ra 0 $, where

$E_0 = \oplus _{1 \leq i \leq n-2}     S(-\ell_{i})$,

$E_1 = \oplus _{1 \leq i_1 < i_2  \leq n-2}         S (-\ell_{i_1} -\ell_{i_2} )$,

\hskip 1cm \vdots

$E_{n-4} = \oplus_{i=1, \ldots , n-2}        S (-\ell_1 - \cdots - \widehat{\ell_i} - \cdots - \ell_{n-2})$,

$E_{n-3} = S (-\ell_1 \ldots -\ell_{n-2})$.

\begin{theorem} \label{th:risoluzione Omega}
The ideal $\tilde\Omega$ has a free resolution of type $0 \rightarrow F_{n-2} \rightarrow F_{n-3} \rightarrow \cdots \rightarrow F_2 \rightarrow F_1 \rightarrow F_0 \rightarrow  \tilde\Omega \rightarrow 0$ with modules

$F_0 = \oplus _{1 \leq i_1<i_2 \leq n} R(-\ell_{i_1} -\ell_{i_2})$,

$F_1 = \oplus _{1 \leq i_1 < i_2 < i_3 \leq n} R^2 (-\ell_{i_1} -\ell_{i_2} -\ell_{i_3})$,

$\vdots$ 

$F_h = \oplus _{1 \leq i_1 < i_2 < \dots < i_{h+2} \leq n}  R^{h+1} (-\ell_{i_1} -\ell_{i_2} -\dots -\ell_{i_{h+2}})$,

$\vdots$

$F_{n-3} = \oplus_{i=1, \ldots , n}  R^{n-2} (-\ell_1 - \cdots - \widehat{\ell_i} - \cdots - \ell_n)$,

$F_{n-2} = R^{n-1} (-\ell_1 \ldots -\ell_n)$,

\noindent which is obtained by a mapping cone procedure that also gives the maps of the resolution by  \eqref{eq:matrice sizigie}.
\end{theorem}

\begin{proof}
We argue by induction on $n$. 
Let 

$\tilde\Omega_1:=\tilde\Omega$, 

$\tilde\Omega_2 :=  (f_1 f_3, \dots, f_1 f_n, \dots, f_{n-1} f_n)$, 

$\vdots$

$\tilde\Omega_{k-1} :=  (f_1 f_k, \dots, f_1 f_n, \dots, f_{n-1} f_n)$, 

$\vdots$

$\tilde\Omega_n := ( f_2 f_3, \dots, f_2 f_n, \dots, f_{n-1} f_n) = \tilde\Omega_{[2]}$.

\noindent For $n=2$ we have $\tilde\Omega=(f_1 f_2)$ and the statement holds because there exists the exact sequence \  
$0 \ra R(-\ell_1 -\ell_2) \ra \tilde\Omega \ra 0$.
Now, assume that the statement holds for $n-1$.

We have $\tilde\Omega=( f_1 f_2, f_1 f_3, \dots, f_1 f_n) + \tilde\Omega_{[2]}$. By the inductive hypothesis, the minimal free resolution of $\tilde\Omega_{[2]}$ is ${\bf F_n}: \ 0 \rightarrow F_{n,n-3}\dots \rightarrow F_{n,1} \rightarrow F_{n,0} \rightarrow 0$, where:

$F_{n,0} = \oplus_{2 \leq i_1<i_2 \leq n} R(-\ell_{i_1} -\ell_{i_2})$,

$F_{n,1} = \oplus_{2 \leq i_1 < i_2 < i_3 \leq n} R^2 (-\ell_{i_1} -\ell_{i_2} -\ell_{i_3})$,

$\vdots$

$F_{n,n-3} = \oplus_{i=2, \ldots , n}  R^{n-2} (-\ell_1 - \cdots - \widehat{\ell_i} - \cdots - \ell_n)$.

\noindent Moreover, the ideal $(\tilde\Omega_n : (f_1 f_n))$ coincides with $\gb _{n} := ( f_2, \dots, f_{n-1})$, so its minimal free resolution ${\bf E_n}$ is \ $0 \ra E_{n-3} \ra E_{n-4} \ra \dots \ra E_1 \ra E_0 \ra \gb \ra 0 $, where

$E_0 = \oplus_{2 \leq i \leq n-1} R(-\ell_{i})$,

$E_1 = \oplus_{2\leq i_1 < i_2  \leq n-1} R(-\ell_{i_1} -\ell_{i_2} )$,

\hskip 1cm \vdots

$E_{n-4} = \oplus_{i=2, \ldots , n-1}  R(-\ell_2 - \cdots - \widehat{\ell_i} - \dots - \ell_{n-1})$,

$E_{n-3} = R (-\ell_2 \ldots -\ell_{n-1})$.
\vskip 1mm
\noindent Thus, the mapping cone of ${\bf E_n}(-\ell_1 - \ell_n) \ra {\bf F_n}$ that is  induced by the multiplication by $f_1 f_n$ gives the following free resolution ${\bf F_{n-1}}$ of $\tilde\Omega_{n-1} = (f_1 f_n) + \tilde\Omega_n$:

$F_{n-1,0}= R(-\ell_1 - \ell_n) \oplus F_{n,0}$.

$F_{n-1,1}= E_{0}(-\ell_1 - \ell_n) \oplus F_{n,1}$. 

$F_{n-1,2}= E_{1}(-\ell_1 - \ell_n) \oplus F_{n,2}$. 

$\hskip 5mm \vdots$

$F_{n-1,n-2}=  E_{n-3}(-\ell_1 - \ell_n) = R(-\ell_2 - \cdots -\ell_{n-1 } -\ell_1 -\ell_n)$.

\noindent Analogously, we successively add the other generators $f_1 f_{n-1}$, $\dots$, $f_1 f_{2}$ and find a free resolution of $\tilde\Omega$ as it is described in the statement. 
\end{proof}

\begin{corollary} \label{cor:congettura}
A complete grid (or pseudo-grid) of type $(\ell_1,\dots,\ell_n)$ in $\mathbb P^n_K$ has minimal free resolution $0 \rightarrow F_{n-2} \rightarrow F_{n-3} \rightarrow \cdots \rightarrow F_2 \rightarrow F_1 \rightarrow F_0 \rightarrow I \rightarrow 0$ like in \eqref{eq:conjecture} with maps given by \eqref{eq:matrice sizigie}.
\end{corollary}

\begin{proof}
Let $A=(L_{j,i})$ be the $1$-lifting (or $1$-pseudo-lifting) matrix determining the ideal $I$ of the given complete grid (or complete pseudo-grid). Then, consider the homogeneous forms $f_1:=\prod_{i=1}^{\ell_1} L_{1,i},\dots,f_n:=\prod_{i=1}^{\ell_n} L_{n,i}$, which is a regular sequence thanks to condition $(\alpha)$ (see also \cite[Proposition 1.7(b)]{CMR2005}.  
It is enough to prove that the resolution \eqref{eq:conjecture} is minimal for these polynomials.

Observe that, if $\beta_h$ and $\beta_{hj}$ denote, respectively, the total Betti numbers and the graded Betti numbers of that free resolution, then it is easy to compute $\sum_j \beta_{hj}= \beta_h=(h+1)\binom{n}{h+2}$, for every $h \in \{0,\dots,n-2\}$.

For every $h \in \{0,\dots,n-2\}$, denote by $\beta_{hj}^I$ the graded Betti numbers and by $\beta_h^I$ the total Betti numbers of $I$.
Thanks to Proposition~\ref{prop:betti grids}, we have $\beta_h^I=(h+1)\binom{n}{h+2}$, and hence $\beta_h=\beta_h^I$.

On the other hand, we have $\beta_h^I = \sum_j \beta^I_{hj}\leq \sum_j \beta_{hj}= \beta_h$, for every $h \in \{0,\dots,n-2\}$. Moreover, we obtain $\beta^I_{hj}\leq \beta_{hj}$, for every graded Betti number, thanks to the well-known fact that the number of minimal generators of a given degree for finitely generated standard graded $K$-algebras is an invariant, by Nakayama Lemma. 
Hence, we obtain $\beta^I_{hj}= \beta_{hj}$, for every graded Betti number, and so the free resolution of $\Omega$ that we have constructed by the mapping cone is minimal.
\end{proof}

We observe that Theorem \ref{th:risoluzione Omega} and Corollary \ref{cor:congettura} do not need that the lifting matrix satisfies condition $(\alpha')$, because condition $(\alpha)$ is sufficient. In fact, although Corollary \ref{cor:congettura} is stated with the assumption that condition $(\alpha')$ is satisfied, this condition is never used in the proof. For example, Theorem \ref{th:risoluzione Omega} and Corollary \ref{cor:congettura} apply also to the case  $f_1=x_1^{\ell_1}, \dots, f_n=x_n^{\ell_n}$. In the Appendix, we describe an example of computation of the minimal free resolution with the maps in this case.


\section{Seminormality of complete pseudo-grids}
\label{sec:seminormality}

The results that have been presented until now are significant for $n\geq 3$, although they also hold for $n=2$. In this section, we consider the property of seminormality which is significant for every $n\geq 2$. 

Recall that an algebraic variety is seminormal at a point $p$ if the local ring $\mathcal O_{X,p}$ is seminormal. If $X$ is seminormal at every point, then $X$ is called seminormal. For definitions and properties of a seminormal ring we refer to \cite{Traverso, GrecoTraverso}.

In \cite[Theorem 26]{GO2007} it is highlighted that a complete grid of lines of type $(\ell_1,\ell_2,\ell_3)$ in $\mathbb P^3_K$, with $\ell_1,\ell_2,\ell_3 \geq 2$, is not seminormal. We now investigate the same problem for complete grid of lines in $\mathbb P^n_K$, for every $n\geq 2$, and observe that a complete pseudo-grid is seminormal if it is obtained by a pseudo-$1$-lifting matrix that satisfies condition ($\alpha'$) and the further next condition. 

\begin{definition}\label{def:condition beta}
For a pseudo-$t$-lifting matrix $A$ we can consider the condition:
\begin{itemize} 
\item[($\beta$)] for any choice of $k$ entries of $A$, say $L_1,\dots,L_k\in R$ from $A$, $\dim_K\langle L_1,\dots,L_k \rangle = \min\{n+t,k\}$ (see \cite[Remark 3.2]{MiNa}).
\end{itemize}
\end{definition}

\begin{remark}\cite[Remark 1.9(c)]{CMR2005}
For $r>t+1$, condition ($\beta$) does not make sense for a $t$-lifting matrix.
\end{remark}

\begin{example}
The $1$-pseudo-lifting matrix \eqref{eq:pseudo} considered in Example \ref{ex:pseudo lifting} satisfies condition $(\alpha')$ but does not satisfy condition $(\beta)$, because the vector space $\langle x_2-3x_0, x_2+x_1+3x_0, x_3-3x_0, x_3-3x_2-x_1\rangle$ has dimension $3<4$. 
\end{example}

\begin{proposition}\label{prop:seminormality}
\vspace{-2mm}
\begin{itemize}
\item[(i)]  A complete grid of lines $Y\subset\mathbb P^n_K$ of type $(\ell_1,\dots,\ell_n)$, with $\ell_1\dots \ell_{k-1}\ell_{k+1}\dots \ell_n> n$ for some $k$, is not seminormal, for every $n\geq 2$. 
\item[(ii)]  A complete pseudo-grid of lines in $\mathbb P^n_K$ that is obtained by a pseudo-$1$-lifting matrix satisfying conditions ($\beta$) is a seminormal configuration of lines, for every $n\geq 2$.
\end{itemize}
\end{proposition}

\begin{proof}
For a complete grid of lines, every singular point is an isolated singular point and intersection of lines. Then we can use the characterization that has given in \cite[Theorem~2.1]{COR2012} on the base of results in \cite[Theorems 2.9 and 3.1]{O76}. 

For what concerns item {\em (i)}, from \cite[Theorem 2.1]{COR2012} we have that a complete grid of lines $Y$ is seminormal if and only if every intersection of $n+1$ lines of $Y$ is empty. However, through each point at infinity $p_{\infty,k}$ there are at least $n+1$ lines of the complete grid, by the hypothesis. If $n=3$ we can also refer to \cite[Theorem~26]{GO2007}.

Item {\em (ii)} is \cite[Corollary 4.7]{CMR2005} applied to those  complete pseudo-grids of lines that are determined by a $1$-pseudo-lifting matrix satisfying condition $(\beta)$.
\end{proof}

\begin{example}\label{ex:non seminormale}
From \cite[Example 2.1]{GOR2014}, consider the configuration $C$ of $8$ lines that is obtained by eliminating from the edges of a cube in $\mathbb P^3_K$ four lines parallel to a same coordinate axis. The cube is a complete grid of type $(2,2,2)$ and $C$ can be obtained by the $1$-lifting of the unmixed monomial ideal $J=(x_1^2,x_2x_3)$ induced by the matrix $A=\left(\begin{array}{cc}x_1 & x_1-x_0\\ x_2 & x_2-x_0 \\ x_3 & x_3-x_0\end{array}\right)$. From Proposition \ref{prop:seminormality}, $C$ is not seminormal. Nevertheless, a pseudo-$1$-lifting of $J$ by a matrix satisfying condition ($\beta$) is seminormal.  
\end{example}

\section*{Acknowledgment}

We thank the referee very much for his/her suggestions and comments, which have been crucial for establishing connections and making comparisons with related areas of research.  

The first author is member of “National Group for Algebraic and Geometric
Structures, and their Applications” (GNSAGA-INDAM).

\begin{small}
\section*{Appendix}

We give a detailed description of the three mapping cones that are needed in order to obtain the maps of the minimal free resolution of the monomial ideal 
$$\Omega=(x_1^{\ell_1}x_2^{\ell_2}, x_1^{\ell_1}x_3^{\ell_3}, x_1^{\ell_1}x_4^{\ell_4}, x_2^{\ell_2}x_3^{\ell_3},x_2^{\ell_2}x_4^{\ell_4},x_3^{\ell_3}x_4^{\ell_4}) = (x_1^{\ell_1}x_2^{\ell_2}, x_1^{\ell_1}x_3^{\ell_3}, x_1^{\ell_1}x_4^{\ell_4})+ \Omega_{[2]},$$ 

\noindent with $ \Omega_{[2]}= (x_2^{\ell_2}x_3^{\ell_3},x_2^{\ell_2}x_4^{\ell_4}) + \Omega_{[3]}$ and  $ \Omega_{[3]} = (x_3^{\ell_3}x_4^{\ell_4})$.\\
Let $f_1:=x_1^{\ell_1}x_2^{\ell_2}$, $f_2:=x_1^{\ell_1}x_3^{\ell_3}$, $f_3:=x_1^{\ell_1}x_4^{\ell_4}$, $f_4:=x_2^{\ell_2}x_3^{\ell_3}$, $f_5:=x_2^{\ell_2}x_4^{\ell_4}$, $f_6:=x_3^{\ell_3}x_4^{\ell_4}$.
According to the construction of Section \ref{sec:mapping cone}, the minimal free resolution of $\Omega_{[2]}$ is
$$0\ra F_1 \stackrel{\varphi_1}{\ra} F_0 \stackrel{\varphi_0}{\ra} \Omega_{[2]} \ra 0$$
where $F_1:=S^2(-\ell_2-\ell_3-\ell_4)$ and $F_0:=S(-\ell_2-\ell_3)\oplus S(-\ell_2-\ell_4) \oplus S(\ell_3-\ell_4)$, and 

$$\begin{array}{ll} 
 M_{\varphi_0}=\Bigl(\begin{array}{ccc} x_2^{\ell_2}x_3^{\ell_3} & x_2^{\ell_2}x_4^{\ell_4} & x_3^{\ell_3}x_4^{\ell_4} \end{array}\Bigr),
&
M_{\varphi_1}=\left(\begin{array}{cc} 
-x_4^{\ell_4}  & 0\\
0 & -x_3^{\ell_3} \\
x_2^{\ell_2} & x_2^{\ell_2}
\end{array}\right).
\end{array}$$

\noindent {\bf First step.} The ideal $(\Omega_{[2]} : f_3) = 
 (x_2^{\ell_2}, x_3^{\ell_3} )=:\mathfrak b_4$ has the minimal free resolution

$0 \ra E_1^{\mathfrak b_4} \stackrel{\delta_1^{\mathfrak b_4}}{\ra} E_0^{\mathfrak b_4} \stackrel{\delta_0^{\mathfrak b_4}}{\ra} \mathfrak b_4 \ra 0$,
where $E_1^{\mathfrak b_4}=S(-\ell_2-\ell_3)$ and $E_0^{\mathfrak b_4}=S(-\ell_2)\oplus S(-\ell_3)$, and 
$$\begin{array}{ll} 
M_{\delta_0}^{\mathfrak b_4}=\Bigl(\begin{array}{cc} x_2^{\ell_2} & x_3^{\ell_3}  \end{array}\Bigr), 
&
M_{\delta_1}^{\mathfrak b_4}=\left(\begin{array}{c} 
-x_3^{\ell_3} \\
x_2^{\ell_2} 
\end{array}\right).
\end{array}$$

\noindent Thus, we obtain the mapping cone ${\bf F}^{\mathfrak b_4}$:
$$0\ra E_1^{\mathfrak b_4}(-\ell_1-\ell_4) \stackrel{\gamma_2^{\mathfrak b_4}}{\ra} E_0^{\mathfrak b_4}(-\ell_1-\ell_4)\oplus F_1 \stackrel{\gamma_1^{\mathfrak b_4}}{\ra} S(-\ell_1-\ell_4) \oplus F_0 \stackrel{\gamma_0^{\mathfrak b_4}}{\ra} (x_1^{\ell_1}x_4^{\ell_4})+\Omega_{[2]}S $$
from
$$\begin{array}{ccccccl}
0 & \ra & E_1^{\mathfrak b_4} (- \ell_1-\ell_4) &  \stackrel{\delta_1^{\mathfrak b_4}}{\ra} & E_0^{\mathfrak b_4} (- \ell_1-\ell_4)&  \stackrel{\delta_0^{\mathfrak b_4}}{\ra} & \mathfrak b_4 (- \ell_1-\ell_4)  \\
  & & \downarrow  \alpha_1^{\mathfrak b_4} & &  \downarrow  \alpha_0^{\mathfrak b_4} & &  \downarrow  \cdot \, x_1^{\ell_1}x_4^{\ell_4} \\
0 &\ra& F_1^{\mathfrak b_4} &  \stackrel{\varphi_1^{\mathfrak b_4}}{\ra} & F_0^{\mathfrak b_4} & \stackrel{\varphi_0^{\mathfrak b_4}}{\ra} & \Omega_{[2]} 
\end{array}$$
with

$$M_{\alpha_0^{\mathfrak b_4}}=\left(\begin{array}{rr} 
 0 & 0 \\
x_1^{\ell_1} & 0 \\
0 & x_1^{\ell_1} 
\end{array}\right), 
M_{\alpha_1^{\mathfrak b_4}}=\left(\begin{array}{r} 0 \\ x_1^{\ell_1} \end{array}\right),  
M_{\gamma_0^{\mathfrak b_4}}=\Bigl(\begin{array}{rrrr} x_1^{\ell_1}x_4^{\ell_4} & x_2^{\ell_2}x_3^{\ell_3} & x_2^{\ell_2}x_4^{\ell_4} & x_3^{\ell_3}x_4^{\ell_4} \end{array}\Bigr)
$$

$$\begin{array}{ll}  
M_{\gamma_1^{\mathfrak b_4}}=\left(\begin{array}{rrrr} 
-x_2^{\ell_2}  & -x_3^{\ell_3} & 0 & 0\\
0 & 0 &- x_4^{\ell_4} & 0 \\
x_1^{\ell_1} & 0 & 0 & -x_3^{\ell_3} \\
0 & x_1^{\ell_1} & x_2^{\ell_2} & x_2^{\ell_2}
\end{array}\right), &
M_{\gamma_2^{\mathfrak b_4}}=\left(\begin{array}{r} x_3^{\ell_3} \\ -x_2^{\ell_2} \\ 0 \\ x_1^{\ell_1}\end{array}\right). 
\end{array}$$

\noindent{\bf Second Step.} The ideal
$((f_3)+\Omega_{[2]}) :( f_2) =
 (x_2^{\ell_2}, x_4^{\ell_4})=: \mathfrak b_3$ has the minimal free resolution
$0 \ra E_1^{\mathfrak b_3} \stackrel{\delta_1^{\mathfrak b_3}}{\ra} E_0^{\mathfrak b_3} \stackrel{\delta_0^{\mathfrak b_3}}{\ra} \mathfrak b_3 \ra 0$,
where $E_1^{\mathfrak b_3}=S(-\ell_2-\ell_4)$ and $E_0^{\mathfrak b_3}=S(-\ell_2)\oplus S(-\ell_4)$, and 
$$\begin{array}{ll} 
M_{\delta_0}^{\mathfrak b_3}=\Bigl(\begin{array}{cc} x_2^{\ell_2} & x_4^{\ell_4}  \end{array}\Bigr), 
&
M_{\delta_1}^{\mathfrak b_3}=\left(\begin{array}{c} 
-x_4^{\ell_4} \\
x_2^{\ell_2} 
\end{array}\right).
\end{array}$$
Thus, we obtain the mapping cone ${\bf F}^{\mathfrak b_3}$:
$$0\ra E_1^{\mathfrak b_3}(-\ell_1-\ell_3)\oplus F_2^{\mathfrak b_4} \stackrel{\gamma_2^{\mathfrak b_3}}{\ra} E_0^{\mathfrak b_3}(-\ell_1-\ell_3) \oplus F_1^{\mathfrak b_4}\stackrel{\gamma_1^{\mathfrak b_3}}{\ra} S(-\ell_1-\ell_3) \bigoplus F_0^{\mathfrak b_4} \stackrel{\gamma_0^{\mathfrak b_3}}{\ra} (x_1^{\ell_1}x_3^{\ell_3}, x_1^{\ell_1}x_4^{\ell_4})+\Omega_{[2]}$$
from

$$\begin{array}{ccccccccl}
0 & \ra & 0 & \ra & E_1^{\mathfrak b_3} (- \ell_1-\ell_3) &  \stackrel{\delta_1^{\mathfrak b_3}}{\ra} & E_0^{\mathfrak b_3} (- \ell_1-\ell_3)&  \stackrel{\delta_0^{\mathfrak b_3}}{\ra} & \mathfrak b_3 (- \ell_1-\ell_3)  \\
& &  & & \downarrow  \alpha_1^{\mathfrak b_3} & &  \downarrow  \alpha_0^{\mathfrak b_3} & &  \downarrow  \cdot \, x_1^{\ell_1}x_3^{\ell_3} \\
0 &\ra& F_2^{\mathfrak b_4} & \stackrel{\varphi_2^{\mathfrak b_4}}{\ra} & F_1^{\mathfrak b_4} &  \stackrel{\varphi_1^{\mathfrak b_4}}{\ra} & F_0^{\mathfrak b_4} & \stackrel{\varphi_0^{\mathfrak b_4}}{\ra} & (x_1^{\ell_1}x_4^{\ell_4})+ \Omega_{[2]} 
\end{array}$$
with $\varphi_i^{\mathfrak b_4} := \gamma_i^{\mathfrak b_4}$ and

$$
M_{\alpha_0^{\mathfrak b_3}}=\left(\begin{array}{cc} 
 0 & 0 \\
x_1^{\ell_1} & 0 \\
0 & 0 \\
0 & x_1^{\ell_1} 
\end{array}\right),
M_{\alpha_1^{\mathfrak b_3}}=\left(\begin{array}{c} 0 \\ 0  \\ x_1^{\ell_1} \\ 0 \end{array}\right),
M_{\gamma_0^{\mathfrak b_3}}=\Bigl(\begin{array}{rrrrr} x_1^{\ell_1}x_3 ^{\ell_3} & x_1^{\ell_1}x_4^{\ell_4} & x_2^{\ell_2}x_3^{\ell_3} & x_2^{\ell_2}x_4^{\ell_4} & x_3^{\ell_3}x_4^{\ell_4} \end{array}\Bigr).
$$

$$\begin{array}{ll} 
M_{\gamma_1^{\mathfrak b_3}}=\left(\begin{array}{cccccc} 
-x_2^{\ell_2}  & -x_4^{\ell_4} & 0 & 0 & 0 & 0\\
0 & 0 & -x_2^{\ell_2} & -x_3^{\ell_3} & 0 & 0 \\
x_1^{\ell_1} & 0 & 0 & 0 & - x_4^{\ell_4} & 0 \\
0 & 0 & x_1^{\ell_1} & 0 & 0 & -x_3^{\ell_3} \\
0 & x_1^{\ell_1} & 0 & x_1^{\ell_1} & x_2^{\ell_2} & x_2^{\ell_2}
\end{array}\right), 
&
M_{\gamma_2^{\mathfrak b_3}}=\left(\begin{array}{cc} x_4^{\ell_4} & 0 \\- x_2^{\ell_2} & 0 \\ 0 &x_3^{\ell_3} \\ 0 & -x_2^{\ell_2} \\ x_1^{\ell_1} & 0 \\ 0 & x_1^{\ell_1}  \end{array}\right).
\end{array}$$

\noindent {\bf Third Step.} The ideal
$$((f_2,f_3)+\Omega_{[2]}) :( f_1) =
(x_3^{\ell_3}, x_4^{\ell_4}) =: \mathfrak b_2$$
has the minimal free resolution 
$0 \ra E_1^{\mathfrak b_2} \stackrel{\delta_1^{\mathfrak b_2}}{\ra} E_0^{\mathfrak b_2} \stackrel{\delta_0^{\mathfrak b_2}}{\ra} \mathfrak b_2 \ra 0$, 
where $E_1^{\mathfrak b_2}=S(-\ell_3-\ell_4)$ and $E_0^{\mathfrak b_2}=S(-\ell_3)\oplus S(-\ell_4)$, and 
$$\begin{array}{ll} 
M_{\delta_1}^{\mathfrak b_2}=\left(\begin{array}{r} 
-x_4^{\ell_4} \\
x_3^{\ell_3} 
\end{array}\right), & M_{\delta_0}^{\mathfrak b_2}=\Bigl(\begin{array}{rr} x_3^{\ell_3} & x_4^{\ell_4}  \end{array}\Bigr).
\end{array}$$
Thus, we obtain the mapping cone ${\bf F}^{\mathfrak b_2}$:
$$0\ra E_1^{\mathfrak b_2}(-\ell_1-\ell_2) \stackrel{\gamma_2^{\mathfrak b_2}}{\ra} E_0^{\mathfrak b_2}(-\ell_1-\ell_2) \oplus F_1^{\mathfrak b_3}\stackrel{\gamma_1^{\mathfrak b_2}}{\ra} S(-\ell_1-\ell_2) \bigoplus F_0^{\mathfrak b_3} \stackrel{\gamma_0^{\mathfrak b_2}}{\ra}$$ $$\stackrel{\gamma_0^{\mathfrak b_2}}{\ra} (x_1^{\ell_1}x_2^{\ell_2}, x_1^{\ell_1}x_3^{\ell_3}, x_1^{\ell_1}x_4^{\ell_4})+\Omega_{[2]}S$$
from
$$\begin{array}{ccccccccl}
0 & \ra & 0 & \ra & E_1^{\mathfrak b_2} (- \ell_1-\ell_2) &  \stackrel{\delta_1^{\mathfrak b_2}}{\ra} & E_0^{\mathfrak b_2} (- \ell_1-\ell_2)&  \stackrel{\delta_0^{\mathfrak b_2}}{\ra} & \mathfrak b_2 (- \ell_1-\ell_2)  \\
& &  & & \downarrow  \alpha_1^{\mathfrak b_2} & &  \downarrow  \alpha_0^{\mathfrak b_2} & &  \downarrow  \cdot \, x_1^{\ell_1}x_2^{\ell_2} \\
0 &\ra& F_2^{\mathfrak b_3} & \stackrel{\varphi_2^{\mathfrak b_3}}{\ra} & F_1^{\mathfrak b_3} &  \stackrel{\varphi_1^{\mathfrak b_3}}{\ra} & F_0^{\mathfrak b_3} & \stackrel{\varphi_0^{\mathfrak b_3}}{\ra} & (x_1^{\ell_1}x_3^{\ell_3}, x_1^{\ell_1}x_4^{\ell_4})+ \Omega_{[2]} 
\end{array}$$
with $\varphi_i^{\mathfrak b_3} := \gamma_i^{\mathfrak b_3}$ and
$$\begin{array}{ll}
M_{\alpha_0^{\mathfrak b_2}}=\left(\begin{array}{rr} 
 0 & 0 \\
 0 & 0 \\
x_1^{\ell_1} & 0 \\
0 & x_1^{\ell_1} \\
0 & 0 
\end{array}\right),
&
M_{\alpha_1^{\mathfrak b_2}}=\left(\begin{array}{r} 0 \\ 0 \\ 0\\ 0 \\  x_1^{\ell_1} \\  -x_1^{\ell_1} \end{array}\right), 
\end{array}$$

$$M_{\gamma_0^{\mathfrak b_2}}=\Bigl(\begin{array}{rrrrrr} x_1^{\ell_1}x_2^{\ell_2} & x_1^{\ell_1}x_3 ^{\ell_3} & x_1^{\ell_1}x_4^{\ell_4} & x_2^{\ell_2}x_3^{\ell_3} & x_2^{\ell_2}x_4^{\ell_4} & x_3^{\ell_3}x_4^{\ell_4} \end{array}\Bigr),
$$

$$\begin{array}{ll} 
M_{\gamma_1^{\mathfrak b_2}}=\left(\begin{array}{cccccccc} 
-x_3^{\ell_3}  & -x_4^{\ell_4} & 0 & 0 & 0 & 0 & 0 & 0\\
0 & 0 & -x_2^{\ell_2} & -x_4^{\ell_4} & 0 & 0 & 0 & 0 \\
0 & 0 & 0 & 0 & -x_2^{\ell_2} & -x_3^{\ell_3} & 0 & 0 \\
x_1^{\ell_1} & 0 & x_1^{\ell_1} & 0 & 0 & 0 & -x_4^{\ell_4} & 0 \\
0 & x_1^{\ell_1} & 0 & 0 & x_1^{\ell_1} & 0 & 0  &  -x_3^{\ell_3} \\
0 & 0 & 0 & x_1^{\ell_1} & 0 & x_1^{\ell_1} &  x_2^{\ell_2} & x_2^{\ell_2}
\end{array}\right), 
&
M_{\gamma_2^{\mathfrak b_2}}=\left(\begin{array}{ccc}
 x_4^{\ell_4} & 0 &0\\
 -x_3^{\ell_3} &0 &0 \\
 0 &x_4^{\ell_4} & 0\\
 0 & -x_2^{\ell_2} & 0 \\
 0 & 0 & x_3^{\ell_3} \\ 
0 & 0 &- x_2^{\ell_2} \\
 x_1^{\ell_1} &  x_1^{\ell_1} & 0 \\ 
- x_1^{\ell_1} & 0  & x_1^{\ell_1} \end{array}\right).
\end{array}$$

We notice that the matrices $M_{\gamma_i^{\mathfrak b_2}}$ that we have just described correspond to the minimal free resolution with modules 

\vskip 1mm
\noindent $H_2:=E_1^{\mathfrak b_2}(-\ell_1-\ell_2) = S^3(-\ell_1 -\ell_2 -\ell_3 -\ell_4)$, 

\vskip 1mm 
\noindent $H_1:=E_0^{\mathfrak b_2}(-\ell_1-\ell_2) \oplus F_1^{\mathfrak b_3} = S(-\ell_1 -\ell_2 -\ell_3) \oplus    S(-\ell_1 -\ell_2 -\ell_4) \oplus   S(-\ell_1 -\ell_2 -\ell_3) \oplus  S(-\ell_1 -\ell_3 -\ell_4) \oplus  S(-\ell_1 -\ell_2 -\ell_4) \oplus  S(-\ell_1 -\ell_3 -\ell_4) \oplus   
S^2(-\ell_2 -\ell_3 -\ell_4)$, 

\vskip 1mm 
\noindent $H_0:= S(-\ell_1-\ell_2) \bigoplus F_0^{\mathfrak b_3} =  S(-\ell_1-\ell_2) \oplus S(-\ell_1-\ell_3) \oplus S(-\ell_1-\ell_4) \oplus S(-\ell_2-\ell_3) \oplus S(-\ell_2-\ell_4) \oplus S(-\ell_3-\ell_4)$. 

\smallskip
If we replace the module $H_1$ by the isomorphic module  

\noindent $G_1:=S^2(-\ell_1 -\ell_2 -\ell_3) \oplus  S^2(-\ell_1 -\ell_2 -\ell_4) \oplus  S^2(-\ell_1 -\ell_3 -\ell_4) \oplus  S^2(-\ell_2 -\ell_3 -\ell_4)$ 

\noindent in order to exactly obtain formula \eqref{eq:conjecture}, then we obviously have
$$0\ra H_2 \stackrel{\gamma_2}{\ra}
 G_1 \stackrel{\gamma_1}{\ra}
 H_0 \stackrel{\gamma_0}{\ra} \Omega$$
with $M'_{\gamma_0}= M_{\gamma_0^{\mathfrak b_2}}$ and  
$$\begin{array}{ll}
M'_{\gamma_2}=\left(\begin{array}{ccc}
 x_4^{\ell_4} & 0 &0\\
 0 &x_4^{\ell_4} &0 \\
 -x_3^{\ell_3} &0 & 0\\
 0 & 0 & x_3^{\ell_3} \\
 0 & -x_2^{\ell_2} & 0 \\ 
0 & 0 &- x_2^{\ell_2} \\
 x_1^{\ell_1} &  x_1^{\ell_1} & 0 \\ 
- x_1^{\ell_1} & 0  & x_1^{\ell_1} \end{array}\right),
&
M'_{\gamma_1}=\left(\begin{array}{cccccccc} 
-x_3^{\ell_3}  & 0 & -x_4^{\ell_4} & 0 & 0 & 0 & 0 & 0\\
0 & -x_2^{\ell_2} & 0 & 0 &- x_4^{\ell_4} & 0 & 0 & 0 \\
0 & 0 & 0  & -x_2^{\ell_2} & 0 & -x_3^{\ell_3} & 0 & 0 \\
x_1^{\ell_1} &  x_1^{\ell_1} & 0 & 0 & 0 & 0 & -x_4^{\ell_4} & 0 \\
0 & 0 & x_1^{\ell_1}  &  x_1^{\ell_1} & 0 & 0 & 0  &  -x_3^{\ell_3} \\
0 & 0 & 0 & 0 & x_1^{\ell_1}  & x_1^{\ell_1} &  x_2^{\ell_2} & x_2^{\ell_2}
\end{array}\right). 
\end{array}
$$
\end{small}


\providecommand{\bysame}{\leavevmode\hbox to3em{\hrulefill}\thinspace}
\providecommand{\MR}{\relax\ifhmode\unskip\space\fi MR }
\providecommand{\MRhref}[2]{%
  \href{http://www.ams.org/mathscinet-getitem?mr=#1}{#2}
}
\providecommand{\href}[2]{#2}

\end{document}